\documentclass[a4paper,draft]{amsproc}
\usepackage{amssymb}
\usepackage{amscd}

\theoremstyle{plain}
 \newtheorem{thm}{Theorem}[section]
 \newtheorem{prop}{Proposition}[section]
 \newtheorem{lem}{Lemma}[section]
 \newtheorem{cor}{Corollary}[section]
\theoremstyle{definition}
 \newtheorem{exm}{Example}[section]
 \newtheorem{rem}{Remark}[section]
\newtheorem{dfn}{Definition}[section]
\newtheorem{assumption}{Assumption}[section]
\newtheorem{conjecture}{Conjecture}[section]
\numberwithin{equation}{section}

\renewcommand{\leq}{\leqslant} \renewcommand{\geq}{\geqslant}
\renewcommand{\setminus}{\smallsetminus}

\newcommand{\A}{\mathrm{ A}}
\newcommand{\R}{\mathbb{R}}
\newcommand{\reg}{\mathrm{reg\;}}
\newcommand{\ddim}{\mathrm{ddim\;}}
\newcommand{\dind}{\mathrm{dind\;}}
\newcommand{\corank}{\mathrm{corank\;}}

\newcommand{\FF}{\mathcal{F}}
\newcommand{\OO}{\mathcal{O}}
\newcommand{\UU}{\mathcal{U}}
\newcommand{\T}{\mathbb{T}}
\newcommand{\g}{\mathfrak{g}}

\DeclareMathOperator{\ann}{\mathrm{ann}}
\newcommand{\ad}{\mathrm{ad}}
\DeclareMathOperator{\Ad}{\mathrm{Ad}}
\DeclareMathOperator{\Span}{\mathrm{span\,}}
\DeclareMathOperator{\diag}{\mathrm{diag}}
\DeclareMathOperator{\rank}{\mathrm{rank}}
\DeclareMathOperator{\pr}{\mathrm{pr}}

\title[Symmetries and Integrability]{SYMMETRIES AND INTEGRABILITY
\footnote{\scriptsize PUBLICATIONS DE L'INSTITUT MATH\'EMATIQUE
Nouvelle s\'erie, tome 84(98) (2008), 1--36} \footnote {DOI:
10.2298/PIM0898001J}}

\author[Jovanovi\'c]{\bfseries Bo\v zidar Jovanovi\'c}

\address{
Mathematical Institute \\
Serbian Academy of Science and Art \\
Kneza Mihaila 36, 11000 Belgrade\\
Serbia}
\email{bozaj@mi.sanu.ac.rs}

\begin{document}


\begin{abstract}
This is a survey on finite-dimensional integrable dynamical systems
related to Hamiltonian $G$-actions.
Within a framework of noncommutative integrability
we study integrability of $G$-invariant systems, collective motions and reduced integrability.
We also consider reductions of the Hamiltonian flows
restricted to their invariant submanifolds generalizing classical Hess--Appel'rot case
of a heavy rigid body motion.
\end{abstract}

\maketitle

\tableofcontents

\section{Introduction}

\subsection{Hamiltonian Systems}
Let $(M,\omega)$ be $2n$-dimensional connected symplectic manifold.
By $X_h$ we shall denote the Hamiltonian vector field of a function $h$
and by $\{\cdot,\cdot\}$ the canonical Poisson bracket on $M$ defined by
$$
dh_x(\xi)=\omega_x(\xi,X_h(x)), \quad \xi\in T_xM \quad \text{and}\quad
\{f_1,f_2\}=\omega(X_{f_2},X_{f_1})=df_1(X_{f_2}),
$$
respectively (we follow the sign convention of Arnold \cite{Ar}).
The equations:
\begin{equation} \label{hamiltonian}
\dot x=X_h(x)\quad \Longleftrightarrow \quad \dot f=\{f,h\}, \quad f\in C^\infty (M)
\end{equation}
are called Hamiltonian equations with the Hamiltonian function $h$.

A function $f$ is an integral of the Hamiltonian system
(constant along trajectories of (\ref{hamiltonian}))
if and only if it commutes with $h$: $\{h,f\}=0$.

\subsection{ Natural Mechanical Systems}
The basic examples of Hamiltonian systems are natural mechanical systems $(Q,\kappa,v)$,
where $Q$ is a configuration space, $\kappa$ is a Riemannian metric on $Q$
(also regarded as a mapping $\kappa: TQ\to T^*Q$) and $v:Q\to\R$ is a potential function.

Let $q=(q^1,\dots,q^n)$ be local coordinates on $Q$.
The motion of the system is described by the Euler--Lagrange equations
\begin{equation} \label{Lagrange}
\frac{d}{dt}\frac{\partial l}{\partial \dot q^i}=\frac{\partial l}{\partial q^i},
\quad i=1,\dots,n,
\end{equation}
where the Lagrangian is
$l(q,\dot q)=\frac12(\kappa_q \dot q,\dot q)-v(q)
=\frac12\sum_{ij}\kappa_{ij}\dot q^i\dot q^j-v(q)$.

Equivalently, we can pass from velocities $\dot q^i$ to the momenta $p_j$
by using the standard Legendre transformation $p_j=\kappa_{ij}\dot q^i$.
Then in the coordinates $q^i, p_i$ of the cotangent bundle $T^*Q$ the equations of motion read:
\begin{equation} \label{1}
\frac{dq^i}{dt}=\frac{\partial h}{\partial p_i},\qquad
\frac{dp_i}{dt}=-\frac{\partial h}{\partial q^i}, \qquad i=1,\dots,n,
\end{equation}
where $h$ is the Legendre transformation of $l$
\begin{equation} \label{ham}
 h(q,p)=\frac{1}{2}\sum_{i,j=1}^n\kappa^{ij}p_i p_j+v(q),
\end{equation}
interpreted as the energy of the system.
Here $\kappa^{ij}$ are the coefficients of the tensor inverse to the metric.

This system of equations is Hamiltonian on $T^*Q$
endowed with the \emph{canonical symplectic form} $\omega=\sum_{i=1}^n dp_i\wedge dq^i$.
The corresponding canonical Poisson bracket is given by
\begin{equation}
\{f,g\}=\sum_{i=1}^{n}\left(
 \frac{\partial f}{\partial q^i}\frac{\partial g}{\partial p_i}
-\frac{\partial g}{\partial q^i}\frac{\partial f}{\partial p_i}\right).
\label{CPB}\end{equation}

For $v\equiv 0$, the equations (\ref{1}) are \emph{the geodesic flow equations}
of the Riemannian manifold $(Q,\kappa)$.

\subsection{ Quadratures}
Recall that integration by quadratures of a system of differential equations $\dot x=X(x)$
in some domain $V\subset\mathbb{R}^n\{x\}$ is the search for its solutions
by a finite number of ``algebraic" operation
(including inversion of the functions) and ``quadratures",
i.e., calculations of the integrals of known functions.

The Lie theorem says that if we have $n$ linearly independent vector fields $X_1,\dots,X_n$
that generate a solvable Lie algebra under commutation and $[X_1,X_i]=\lambda_i X_i$,
then the differential equation $\dot x=X_1(x)$ can be integrated by quadratures in $V$
(e.g., see \cite{AKN}).

By modifying the Lie theorem, Kozlov obtained the following result
which gives the sufficient conditions to the integration by quadratures
of Hamiltonian systems with invariant relations (see \cite{KK, AKN, Kozlov}).

Consider $\R^{2n}=\{(q^1,\dots,q^n,p_1,\dots,p_n)\}$
endowed with the canonical Poisson bracket (\ref{CPB}).
Suppose that we have $n$-functions $f_1,\dots,f_n$
that generate solvable Lie algebra with respect to Poisson bracket
$$
\{f_1,f_j\}=c^1_{1,j}f_1, \;
\{f_2,f_j\}=c^1_{2,j}f_1+c^2_{2,j} f_2, \;\dots,\;
\{f_n,f_j\}=c^1_{n,j}f_1+\cdots+c^n_{n,j} f_n
$$
and which are independent on the level-set
\[
M_c=\{(q,p)\in \R^{2n}\mid f_1=c_1,\dots,f_n=c_n\}.
\]

\begin{thm} [Kozlov \cite{Kozlov}]
Suppose
$$
c^1_{1,j}c_1=0, \;
c^1_{2,j}c_1+c^2_{2,j} c_2=0, \,\, \dots, \,\,
c^1_{n,j}c_1+\dots+c^n_{n,j} c_n=0
$$
for all $j$. Then $M_c$ is an invariant manifolds for Hamiltonian systems
with Hamiltonians $f_1,\dots,f_n$.
The solutions of the Hamiltonian systems that lie on $M_c$ can be found by quadratures.
\end{thm}

In a special case one gets \emph{Liouville's theorem}:
if a Hamiltonian system with $n$ degrees of freedom has $n$ independent first integrals
in involution than (at least locally) it can be integrated by quadratures.

\subsection{ Liouville Integrability}
One of the central problems in Hamiltonian dynamics is
whether the equations (\ref{hamiltonian}) are completely integrable or not.
The usual definition of complete integrability is as follows:

\begin{dfn}
Hamiltonian equations (\ref{hamiltonian}) are called \emph{completely integrable}
if there are $n$ Poisson-commuting smooth integrals $f_1,\dots,f_n$
whose differentials are independent in an open dense subset of $M$.
The set of integrals $\mathcal F=\{f_1,\dots,f_n\}$
is called a \emph{complete involutive (or commutative) set} of functions on $M$.
\end{dfn}

\begin{rem}
The last condition needs to be commented.
The functional independence of the integrals can be meant in three different senses.
The differentials of $f_1,\dots,f_n$ can be linearly independent:

(i) on an open everywhere dense subset,

(ii) on an open everywhere dense subset of full measure,

(iii) everywhere except for a piece-wise smooth polyhedron.

Instead of these conditions one can assume all the functions to be real analytic.
Then it suffices to require their functional independence at least at one point
(we will call such a situation \emph{analytic integrability}).
\end{rem}

If the system is completely integrable,
by Liouville's theorem it can be integrated by quadratures.
Moreover, the global regularity of dynamics in the case of complete integrability
follows from the following classical Liouville--Arnold theorem:

\begin{thm} [Liouville--Arnold \cite{Liouville, Ar}]
Suppose that the equations (\ref{hamiltonian}) have $n$ Poisson-commuting
smooth integrals $f_1,\dots,f_n$ and let $M_c=\{f_1=c_1,\dots, f_n=c_n\}$
be a common invariant level set.

\emph{(i)} If $M_c$ is regular (the differentials of $f_1,\dots,f_n$ are independent on it),
compact and connected, then it is diffeomorphic to the $n$-dimensional Lagrangian torus.

\emph{(ii)} In a neighborhood of $M_c$ there are action-angle variables $I,\varphi\mod2\pi$
such that the symplectic form becomes
$
\omega=\sum_{i=1}^n dI_i\wedge d\varphi_i
$
and the Hamiltonian function $h$ depends only on actions $I_1,\dots,I_n$.
Thus the Hamiltonian equations are linearized
\[
\dot \varphi_1=\omega_1(I)=\frac{\partial h}{\partial I_1}, \dots,
\dot \varphi_r=\omega_r(I)=\frac{\partial h}{\partial I_n},
\]
i.e., dynamics on invariant tori $I_i=const$ is quasi-periodic.
\end{thm}

The topology of toric foliation in many low-dimensional integrable systems is well known
(e.g., see \cite{TF, Os, BF, Au} and references therein).
The obstructions to the existence of global action-angle variables are studied in \cite{Du}.

Even the system is integrable, the dynamics on the singular set
(where the differentials of integrals $f_1,\dots,f_n$ are dependent)
can be quite complicated.
For example, the flows given in \cite{Bu, BoTa, Bu2} are integrable by smooth integrals,
and by Taimanov's theorem \cite{Ta} can not be integrable by analytic ones.

\subsection{ Outline of the paper}
This paper is a survey in integrable dynamical systems related to Hamiltonian $G$-actions
including integrability of collective motions, reduced integrability and partial integrability.
It turns out that most natural framework to use here
is non-commutative integrability developed by Mischenko and Fomenko \cite{MF2} and Nekhoroshev \cite{N}.

In the case of non-commutative integrability
(sometimes known as superintegrability or degenerate integrability)
invariant manifolds are isotropic but non-Lagrangian.
In section 2 we show that non-commutative integrability
always implies usual Liouville integrability by means of smooth commuting function
(see \cite{BJ2}).

In section 3 we recall on the basic notion of Hamiltonian $G$-actions,
Marsden--Weinstein reduction of symmetries
and Lagrange--Routh reduction of natural mechanical systems.

Following \cite{BJ2}, the construction of integrable systems
related to Hamiltonian $G$-actions is presented in section 4.
We also show that integrability of the reduced systems
is essentially equivalent to the integrability of the original systems in section 5
(see \cite{Zu, Jo1}).

In section 6 we study reductions of the Hamiltonian flows
restricted to their invariant submanifolds (see \cite{Jo2}).
If the reduced system is integrable, the original system may not be.
The typical situation is that invariant submanifold is foliated on invariant tori,
but not with quasi-periodic motions.
We refer to such systems as partially integrable.
The classical example is the Hess--Appel'rot case of a heavy rigid body motion
about a fixed point \cite{He, App}.

Many interesting problems related to the torus actions,
algebraic complete integrability as well as separation of variables are not discuss here.

\section{Noncommutative Integrability}

\subsection{Noncommutative Integration}
If $f_1$ and $f_2$ are integrals of (\ref{hamiltonian}),
then so are an arbitrary smooth function $F(f_1,f_2)$ and the Poisson bracket $\{f_1,f_2\}$.
Therefore without loss of generality we can assume that integrals of (\ref{hamiltonian})
form an algebra $\mathcal F$ with respect to the Poisson bracket.
For simplicity we shall assume that $\mathcal F$
is functionally generated by functions $f_1,\dots,f_l$ so that
\[
\{f_i,f_j\}=a_{ij}(f_1,\dots,f_l).
\]

Suppose that
\begin{gather*}
\dim F_x=\dim \Span\{df_i(x)\, \vert\, i=1,\dots,l\}=l, \quad x\in U, \\
\dim \ker\{\cdot,\cdot\}|_{F_x}=r, \quad x\in U,
\end{gather*}
holds for an open dense set $U\subset M$.
Let $\phi: M\to \mathbb{R}^l$ be the moment mapping:
$\phi(x)=(f_1(x),\dots,f_l(x))$ and let $\Sigma=\phi(M\setminus U)$.
Then we have the following noncommutative integration theorem
(see Mish\-chenko and Fomenko \cite{MF2}, Nekhoroshev \cite{N} and Brailov \cite{Br1}):

\begin{thm}Suppose that:
$\dim F_x+\dim \ker\{\cdot,\cdot\}|_{F_x}=\dim M$, for $x\in U$.
Let $c\in\phi(M)\setminus \Sigma$ be a regular value of the moment map. Then:

\emph{(i)} $M_c=\phi^{-1}(c)$ is an isotropic submanifold of $M$
and the equations \eqref{hamiltonian} on $M_c$ can be (locally) solved by quadratures;

\emph{(ii)} Compact connected components $T^r_c$ of $M_c$
are diffeomorphic to $r$-dimen\-sional tori.
The dynamics on $T^r_c$ is quasi-periodic,
i.e., can be linearized in appropriate angle coordinates $\varphi_1,\dots,\varphi_r$:
$$
\varphi_1(t)=\omega_1 t,\ldots,\varphi_r(t)=\omega_r t.
$$
\end{thm}

Let us just point out some steps in the proof of the theorem.
Under the hypotheses of Theorem 2.1,
there exist $r$ linearly independent commuting vector fields $X_1=X_h,X_2,\dots,X_r$ on $M_c$.
They can be obtained as linear combinations
of the Hamiltonian vector fields $X_{f_1},\dots,X_{f_l}$
from the conditions $\omega(X_i,X_{f_j})=0$, $i=1,\dots,r$, $j=1,\dots,l$.
Since commutative algebras are solvable we can apply the Lie theorem
to locally integrate by quadratures the system $\dot x=X_h(x)$.
On a compact connected component $T^r_c$ of $M_c$,
the vector fields $X_1,\dots,X_r$ are complete.
Therefore $T^r_c$ is diffeomorphic to an $r$-dimensional torus
and the motion on the torus is quasi-periodic
(the proof of this fact is just the same as in the usual Liouville theorem).

Recently, the obstructions for noncommutative integration of Hamiltonian systems
within the framework of differential Galois theory have been studied in \cite{MacPr, MPY}.

It is important to note that the assumptions of Theorem 2.1
imply integrability of the considered system in the usual, commutative sense.

\begin{thm} \cite{BJ2}
Under the assumptions of Theorem $2.1$, the Hamiltonian system \eqref{hamiltonian}
is Liouville integrable, i.e., it admits $n$ Poisson-commuting $C^\infty$-smooth integrals
$g_1,\dots,g_n$, independent on an open dense subset of $M$.
\end{thm}

\begin{proof}
On the image of $M$ under the moment mapping $\phi(M)\subset\mathbb{R}^l\{y_1,\dots,y_l\}$
we can introduce a ``Poisson structure" $\{\cdot,\cdot\}_\mathcal F$ by the formulas:
$$
\{y_i,y_j\}_{\mathcal F}=a_{ij}(y_1,\dots,y_l).
$$
Note that $\rank\{\cdot,\cdot\}_{\mathcal F}(y)=l-r=2q$
for all $y\in \phi(M)\setminus \Sigma$
and that $(\phi(M)\setminus\Sigma,\{\cdot,\cdot\}_{\mathcal F})$ is a Poisson manifold.

From the definition of the Poisson bracket $\{\cdot,\cdot\}_{\mathcal F}$
it follows that if smooth functions $F,G: \mathbb{R}^l\to\mathbb{R}$
are in involution with respect to $\{\cdot,\cdot\}_{\mathcal F}$ on $\phi(M)$,
then their liftings $f=F\circ\phi$ and $g=G\circ\phi$ commute on $M$:
\begin{equation} \label{1.4}
\{f,g\}(x)=\{F,G\}_{\mathcal F} (\phi(x))=0.
\end{equation}

Let $z\in \phi(M)\setminus\Sigma$.
By the theorem on the local structure of Poisson brackets (see \cite{Wa, LM}),
there is a neighborhood $U(z)\subset \phi(M)\setminus\Sigma$ of $z$ and $l$
independent smooth functions $G_1,\dots,G_l: U(z)\to\mathbb{R}$, $G_i(z)=0$ such that
\[
\{G_i,G_{i+q}\}_{\mathcal F}=1=-\{G_{i+q},G_i\}_{\mathcal F} \quad i=1,\dots, q
\]
and the remaining Poisson brackets vanish.
Let a ball $B^\alpha(\epsilon_\alpha)$ belong to $U(z)$, where
$
B^\alpha(\epsilon_\alpha)=\{y\in U(z)\mid G_1^2+\dots+G_l^2<\epsilon_\alpha\}.
$
Then, starting from the $n$ involutive functions on $B^\alpha$:
\[
    h_1=G_1^2+G_{1+q}^2,\dots,\;
    h_q=G_q^2+G_{2q}^2,\;
h_{q+1}=G_{2q+1}^2,\dots,\;
    h_n=G_l^2
\]
we can construct a smooth set of nonnegative functions
$F^\alpha_1,\dots,F^\alpha_n:\mathbb{R}^l\to\mathbb{R}$
that are independent on an open dense subset of $B^\alpha(\epsilon)$,
equal to zero outside $B^\alpha(\epsilon_\alpha)$, in involution on $\phi(M)$,
and satisfies inequalities $F_i^\alpha<e^{-\epsilon_\alpha}$ together with all derivatives.
To this end we use the construction suggested by Brailov for Darboux symplectic balls
(see \cite{TF}).

Let $g:\mathbb{R}\to\mathbb{R}$ be smooth nonnegative function,
such that $g(x)$ is equal to zero for $|x|>\epsilon_\alpha$,
monotonically increases on $[-\epsilon,0]$ and monotonically decreases on $[0,\epsilon]$.
Let $h(y)=g(h_1(y)+\dots+h_n(y))$.
This function could be extended by zero to the whole manifold.
Now, we can define $F^\alpha_i$ by: $F^\alpha_i=h\cdot h_i$.
Obviously, $\{F^\alpha_i,F^\alpha_j\}_{\mathcal F}=0$.
These functions are independent inside $B^\alpha$.
Also, we can choose $g$ in such a way that $F^\alpha_i$
satisfies the inequalities $F_i^\alpha<e^{-\epsilon_\alpha}$ together with all derivatives.

In the same way we can construct a countable family of open balls
$\{B^\alpha(\epsilon_\alpha)\}$, $B^\alpha \cap B^\beta=\emptyset$,
and functions $\{F^\alpha_1,\dots,F^\alpha_l\}$ with the above properties,
such that $B=\bigcup_\alpha B^\alpha(\epsilon_\alpha)$
is an open everywhere dense set of $\phi(M)\setminus\Sigma$.
Let us define the functions $F_1,\dots,F_n: \mathbb{R}^l\to \mathbb{R}$ by:
\[
F_i(y)=\begin{cases}
F_i^\alpha,  &y\in B^\alpha\subset B \\
0, &y\in\mathbb{R}^l\setminus B,\quad i=1,\dots,n
\end{cases}
\]

By (\ref{1.4}), the functions $g_1=F_1\circ\phi,\dots,g_n=F_n\circ\phi$
will have the desired properties.
\end{proof}

Theorem 2.2 says that the $r$-dimensional invariant tori $T^r$
can be organized into larger, $n$-dimensional Lagrangian tori $T^n$
that are level sets of a commutative algebra of integrals.
Since the tori $T^n$ are fibered into invariant tori $T^r$,
the trajectories of (\ref{hamiltonian}) are not dense on $T^n$.
In this sense, the system (\ref{hamiltonian}) is degenerate.
Thus, establishing the fact of non-commutative integrability of the system
give us more information on the behavior of its integral trajectories
than we could obtain from the usual Liouville integrability.
Note that, contrary to the case of non-degenerate integrable systems,
the fibration by Lagrangian tori is neither intrinsic nor unique.

Let $P_0=\phi(M)\setminus\Sigma$, $M_0=\phi^{-1}(P_0)$.
If all invariant submanifolds $M_c=\phi^{-1}(c)$, $c\in P_0$ are compact and connected,
then $\phi: (M_0,\{\cdot,\cdot\})\to(P_0,\{\cdot,\cdot\}_{\mathcal F})$
is a Poisson morphism and isotropic fibration.
This fibration is symplectically complete,
i.e., the symplectic orthogonal distribution to the tangent spaces
of the fibres is a foliation (see \cite{LM}).

\subsection{Generalized Action-Angle Variables}
With notations of Theorem 2.1,
let $T^r_c$ be a compact connected component of the invariant level set $M_c$.

\begin{thm} [Nekhoroshev \cite{N}]
In a neighborhood of $T^r_c$
there are \emph{generalized action-angle variables} $p,q,I,\varphi\mod2\pi$,
defined in a toroidal domain $\OO=\T^r\{\varphi\}\times B_\sigma\{I,p,q\}$,
$$
B_\sigma=\biggl\{(I_1,\dots,I_r,p_1,\dots,p_k,q_1,\dots,q_k)\in\R^l
\,\Big|\,\sum_{i=1}^r I_i^2+\sum_{j=1}^k q_j^2+p_j^2\leq\sigma^2\biggr\}
$$
such that the symplectic form becomes
$$
\omega=\sum_{i=1}^r dI_i\wedge d\varphi_i+\sum_{i=1}^k dp_i\wedge dq_i,
$$
and the Hamiltonian function $h$ depends only on $I_1,\dots,I_r$.
The Hamiltonian equations take the following form in action-angle coordinates:
\begin{equation} \label{local}
\dot \varphi_1=\omega_1(I)=\frac{\partial h}{\partial I_1}, \dots,
\dot \varphi_r=\omega_r(I)=\frac{\partial h}{\partial I_r}, \quad
\dot I=\dot p=\dot q=0.
\end{equation}
\end{thm}

The obstructions to the existence of global generalized action-angle variables
are studied in \cite{N, DD, FF}.
The action-angle variables in the case of non-compact invariant manifolds
are given in \cite{FGS}.

\begin{dfn} \cite{BoOI}
The Hamiltonian system (\ref{local})
defined in the toroidal domain $\OO=\T^r\times B_\sigma$ is \emph{$\T^r$-dense}
if the set of points $(I_0,p_0,q_0)\in B_\sigma$
for which the trajectories of (\ref{local}) are dense on the torus $\{I=I_0,p=p_0,q=q_0\}$
is everywhere dense in the ball $B_{\sigma}$.
\end{dfn}

The frequencies $\omega_1(I),\dots,\omega_r(I)$
corresponding to the dense trajectories (\ref{local}) are rationally independent.
For example, a non-degenerate system ($\det(\frac{\partial\omega}{\partial I})\neq 0$
on an open dense set of $\OO$) is $\T^r$-dense.

Any smooth first integral of a $\T^r$-dense system is a function of the variables $I,q,p$ only.

Now, let $h=h(I_1,\dots,I_r)$ be an arbitrary differentiable Hamiltonian function
defined on a toroidal domain $\OO=\T^r\{\varphi\}\times B_\sigma\{I,p,q\}$.
Then we have

\begin{thm} [Bogoyavlenski \cite{BoOI}]
There exists a family of balls $B_\tau\subset B_\sigma$
such that the union $\bigcup_\tau B_\tau$ is dense in $B_\sigma$
and the following properties hold.

\emph{(i)} In the toroidal domain $\OO_\tau=\T^r\times B_\tau$
there exists a canonical transformation:
$\{I,p,q,\varphi\}\to\{I^\tau,p^\tau,q,^\tau,\varphi^\tau\}$,
that transforms the system \eqref{local} to the form
\begin{gather*}
\dot\varphi_1^\tau=\omega_1^\tau=\frac{\partial h}{\partial I_1^\tau},\;\dots,\;
\dot\varphi_{\tilde r(\tau)}^\tau=\omega_{\tilde r}^\tau
=\frac{\partial h}{\partial I_{\tilde r}^\tau},
\\
\dot\varphi_{\tilde r+1}^\tau=0,\dots,\dot\varphi_r^\tau=0, \quad
\dot I^\tau=\dot p^\tau=\dot q^\tau=0,
\end{gather*}
$h=h(I_1^\tau,\dots,I_{\tilde r})$, $\tilde r=\tilde r(\tau)\leq r$.
The system is $\T^{\tilde r}$-dense in $\OO_\tau$
(regarded as the product $\T^{\tilde r}\times (\T^{r-\tilde r}\times B_{\tau})$).

\emph{(ii)} Moreover, if $h(I_1,\dots,I_r)$ is analytic
and the maximal dimension of the closures of trajectories is equal to $\tilde r$,
then such canonical transformation exists globally
and the system is $\T^{\tilde r}$-dense in $\OO$.
\end{thm}

\subsection{Complete Algebras of Functions}
Let $\mathcal F$ be an algebra of functions, closed under the Poisson brackets,
on the symplectic manifold $(M,\omega)$.
Let $F_x$ be the subspace of $T^*_xM$ generated by differentials of functions in $\FF$
and let $K_x\subset F_x$ be the kernel of Poisson structure restricted on $F_x$:
\begin{equation} \label{FK}
F_x=\Span\{df_(x)\mid f\in\FF\}, \quad K_x=\ker\{\cdot,\cdot\}\big|_{F_x}.
\end{equation}

Suppose that $\dim F_x=l$, $\dim K_x=r$ holds on an open dense subset $U \subset M$.
We shall denote $U$ by $\reg{\mathcal F}$ or by $\reg M$
(\emph{regular points} of $\mathcal F$).
The numbers $l$ and $r$ are usually denoted by $\ddim \mathcal F$ and $\dind\mathcal F$
and are called \emph{differential dimension} and \emph{differential index} of $\mathcal F$.

\begin{dfn}
The algebra $\mathcal F$ is said to be \emph{complete} if:
\begin{equation} \label{CAS}
\ddim{\mathcal F}+\dind{\mathcal F}=\dim M.
\end{equation}
\end{dfn}

Note that instead of algebras one usually consider sets of functions
closed under the Poisson bracket.
The notions of completeness, $\ddim$ and $\dind$ are defined just in the same way as above.

For example, suppose that independent functions $f_1,\dots,f_l$
generate a finite dimensional Lie algebra
$\mathfrak g=\mathcal F=\bigoplus_{i=1}^l\mathbb{R}f_i$ under the Poisson brackets:
$$
\{f_i,f_j\}=\sum_{k=1}^l c^k_{ij}f_k,
$$
$c^k_{ij}$ are constants.
Then the numbers $\ddim \mathcal F$ and $\dind \mathcal F$ coincide with the dimension
and the index of the Lie algebra $\mathfrak g$.

By Theorem 2.1, we can give the following definition:

\begin{dfn}
The Hamiltonian system (\ref{hamiltonian})
is \emph{completely integrable in the noncommutative sense}
if it possesses a complete algebra $\mathcal F$ of integrals.
\end{dfn}

In Definition 2.2 we do not require
that $\mathcal F$ is generated by $l=\ddim\mathcal F$ functions.
We shall briefly explain this. Let $x_0$ belong to $\reg\mathcal F$.
Then there are integrals $f_1,\dots,f_l\in \mathcal F$ that are independent in $x_0$,
where $l=\ddim \mathcal F$.
Let $V$ be the open set where the functions $f_1,\dots,f_l$ are independent.
Since $f_1,\dots,f_l$ are integrals of (\ref{hamiltonian})
the trajectory of (\ref{hamiltonian}) that has initial position in $V$ remains in $V$.
Indeed, the phase flow of (\ref{hamiltonian})
preserves the form $df_1\wedge\dots\wedge df_l$.
Therefore we can consider the restriction of (\ref{hamiltonian}) to $V$
and apply Theorem 2.1 to integrate it.

\begin{dfn}
Let $\mathcal F$ be an algebra of functions.
We shall say that $\mathcal A\subset \mathcal F$ is \emph{a complete subalgebra} if
$
\ddim{\mathcal A}+\dind{\mathcal A}=\ddim{\mathcal F}+\dind{\mathcal F}.
$
\end{dfn}

Mishchenko and Fomenko stated the conjecture
that non-commutative integrable systems are integrable in the usual commutative sense
by means of integrals that belong to the same functional class
as the original non-commutative algebra $\FF$ of integrals
(in other words, one can find complete commutative subalgebra $\mathcal A$ of $\FF$).
In the analytic case, when $\mathcal F=\Span_{\mathbb R}\{f_1,\dots,f_l\}$
is a finite-dimensional Lie algebra,
the conjecture has been proved by Mishchenko and Fomenko
in the semisimple case \cite{MF1} and just recently by Sadetov \cite{Sad}
for arbitrary Lie algebras (see also \cite{VY}).

Thus, according to Theorem 2.2, the following general conjecture remains:

\begin{conjecture}
Suppose that on a real-analytic symplectic $2n$-dimensional manifold $M$
we have a Hamiltonian system $\dot x=X_h(x)$ completely integrable
by means of a infinitely dimensional non-commutative algebra $\mathcal F$
of integrals which are real analytic functions.
Then the system possesses $n$ commuting real analytic integrals.
\end{conjecture}

\subsection{ Noncommutative Integrability on Poisson Manifolds}
One can easily formulate the above setting for Hamiltonian systems $\dot f=\{f,h\}$
on a Poisson manifold $(M,\{\cdot,\cdot\})$.
Instead of (\ref{CAS}), the algebra of functions $\mathcal F$ is said to be complete if
$
\ddim\FF+\dind\FF=\dim M+\corank\{\cdot,\cdot\}.
$
Also, in this case, regular compact connected components level-sets of functions
within $\FF$ are tori of dimension $\dim M-\ddim\FF=\dind\FF-\corank\{\cdot,\cdot\}$.

The equivalent definition of the completeness is as follows.
Let $\Lambda$ be the bivector field on $M$ associated to the Poisson structure
$
\{f_1,f_2\}(x)=\Lambda_x(df_1(x),df_2(x)).
$
We say that $\mathcal F$ is \emph{complete at $x$} if the space $F_x$
defined by (\ref{FK}) is coisotropic:
\begin{equation} \label{F-LAMBDA}
F^\Lambda_x \subset F_x.
\end{equation}
Here $F^\Lambda_x$ is skew-orthogonal complement of $F_x$ with respect to $\Lambda$:
$$
F^\Lambda_x=\{X\in T^*_x M \mid \Lambda_x(F_x,X)=0\}.
$$

The algebra $\mathcal F$ is \emph{complete } if it is complete at a generic point $x\in M$.
It is clear that in this case $K_x=F^\Lambda_x$, for a generic $x\in M$.

\begin{rem}\label{SNI}
Specially, for symplectic manifolds it is natural to state the
completeness criterium in terms of the symplectic structure.
Set
\[
W_x=\{X_f(x) \mid f\in {\FF}\}, \quad
D_x=\{ X \in T_x M \mid (F_x,X)=0 \}.
\]
The condition (\ref{F-LAMBDA}) is equivalent to the coisotropy
of $W_x$ ($W_x^\omega \subset W_x$) and isotropy of $D_x$ ($D_x^\omega \supset D_x$)
in the symplectic linear space $T_xM$.
\end{rem}

\section{Symmetries and Reductions}

\subsection{Hamiltonian $G$-Actions}
Let a connected Lie group $G$ act on $2n$-dimen\-sional connected symplectic manifold $(M,\omega)$.
The action is \emph{Hamiltonian} if $G$ acts on $M$ by symplectomorphisms
and there is a well-defined momentum mapping:
\begin{equation} \label{moment_map}
\Phi: M \to \mathfrak g^*
\end{equation}
($\mathfrak g^*$ is a dual space of the Lie algebra $\mathfrak g$)
such that one-parameter subgroups of symplectomorphisms
are generated by the Hamiltonian vector fields of functions
\begin{equation} \label{one}
\phi_\xi(x)=(\Phi(x),\xi), \quad \xi\in\mathfrak g
\end{equation}
and
\begin{equation} \label{com}
\phi_{[\xi_1,\xi_2]}=\{\phi_{\xi_1},\phi_{\xi_2}\}.
\end{equation}
Then $\Phi$ is equivariant with respect to the given action of $G$ on $M$
and the co-adjoint action of $G$ on $\mathfrak g^*$:
$$
\Phi(g\cdot x)=\Ad_g^*(\Phi(x)).
$$
In particular, if $\mu$ belongs to $\Phi(M)$,
then the co-adjoint orbit ${\mathcal O}(\mu)$ belongs to $\Phi(M)$ as well.

As the important example, consider the $G$ action on the configuration space $Q$.
The action can be naturally extended to the Hamiltonian action on $(T^*Q,\omega)$:
$$
g\cdot (q,p)=(g\cdot q, (dg^{-1})^*p)
$$
with the momentum mapping $\Phi$ given by (e.g., see \cite{LM})
\begin{equation} \label{ctg-mm}
(\Phi(q,p),\, \xi)=(p,\, \xi_q), \quad\xi\in\g.
\end{equation}
Here $\xi_q$ is the vector given by the action of one-parameter subgroup $\exp(t\xi)$ at $q$.

\subsection{ Symplectic Reduction}
Now, let $G$ be a Lie group with a free and proper Hamiltonian action
on a symplectic manifold $(M,\omega)$ with the momentum mapping (\ref{moment_map}).
Assume that $\eta$ is a regular value of $\Phi$,
so that $M_\eta=\Phi^{-1}(\eta)$ and $M_{\mathcal O_\eta}=\Phi^{-1}(\mathcal O_\eta)$ are smooth manifolds.
Here $\mathcal O_\eta=G/G_\eta$ is the coadjoint orbit of $\eta$.
The manifolds $M_\eta$ and $M_{\mathcal O_\eta}$
are $G_\eta$-invariant and $G$-invariant, respectively.

There is a unique symplectic structure $\omega_\eta$ on $N_\eta=M_\eta/G_\eta$ satisfying
\[
\omega|_{M_\eta}=d\pi_\eta^*\omega_\eta
\]
where $\pi_\eta: M_\eta\to N_\eta$ is the natural projection
(Marsden and Weinstein \cite{MaWe}).

According to Noether's theorem, if $h$ is a $G$-invariant function,
then the momentum mapping $\Phi$ is an integral of the Hamiltonian system (\ref{hamiltonian}).
In addition, its restriction to the invariant submanifold $M_\eta$
projects to the Hamiltonian system
\begin{equation} \label{reduced-eta}
\dot y=X_{H_\eta}
\end{equation}
on the reduced space $N_\eta$ with $H_\eta$ defined by
\begin{equation} \label{induced_ham}
h|_{M_\eta}=H_\eta\circ\pi_\eta.
\end{equation}

\subsection{ Poisson Reduction}
An alternative description of the reduced space is as follows.
Let $\{\cdot,\cdot\}$ be the canonical Poisson bracket on $(M,\omega)$.
Then the manifold $M/G$ carries the induced Poisson structure $\{\cdot,\cdot\}^G$
defines as follows. Let
$$
\pi: M\to M/G
$$
be the natural projection.
In what follows, by capital letters we shall denote the functions on $M/G$
and with small letters corresponding $G$-invariant functions on $M$:
\[
f=F\circ\pi.
\]
Then
\[
\{F_1,F_2\}^G \circ \pi=\{f_1,f_2\}.
\]

The mapping $N_\eta \to M_{\mathcal O_\eta}/G$
which assigns to the $G_\eta$-orbit of $x\in M_\eta$ the $G$-orbit
through $x$ in $M_{\mathcal O_\eta}$ establish the symplectomorphism between $N_\eta$
and $M_{\mathcal O_\eta}/G$ (e.g., see \cite{OR}).
In particular, $M_{\mathcal O_\eta}/G$ is the symplectic leaf in $(M/G, \{\cdot,\cdot\}^G)$.

The Hamiltonian flow (\ref{hamiltonian}) projects to the Hamiltonian flow
\begin{equation} \label{reduced-flow}
\dot F=\{F,H\}^G, \qquad F\in C^\infty (M/G)
\end{equation}
on the reduced space $M/G$.
The system (\ref{reduced-eta}) is actually the restriction of (\ref{reduced-flow})
to the symplectic leaf $M_{\mathcal O_\eta}/G$.

\subsection{ Cotangent Bundle Reduction}
Let $G$ be a connected Lie group acting freely and properly on $Q$
and $\rho: Q\to B=Q/G$ be the canonical projection.
Then $0$ is the regular value of the cotangent bundle momentum mapping (\ref{ctg-mm})
and the reduced space $(\Phi^{-1}(0)/G,\omega_0)$
is symplectomorphic to $T^*(Q/G)$, e.g., see \cite{OR}.

Now, let $(Q,\kappa,v)$ be a $G$-invariant natural mechanical system.
That is $G$ acts by isometries and the potential is the pull back $v=V\circ\rho$
of the potential $V$ defined on $Q/G$.
Therefore, the corresponding Hamiltonian function (\ref{ham})
is $G$-invariant and the system can be reduced to $T^*(Q/G)$.
The reduced system is also natural mechanical system on $Q/G$ with the potential $V$
and metric which has the clear geometrical interpretation.

Let
\begin{equation} \label{vert}
\mathcal V_q=\{\xi_q \mid \xi\in\g\}
\end{equation}
be the tangent space to the fibber $G\cdot q$ (\emph{vertical space at} $q$)
and $\mathcal V=\bigcup_q\mathcal V_q$ be the vertical distribution.

Consider the \emph{horizontal distribution} $\mathcal H=\bigcup_q\mathcal H_q\subset TQ$
orthogonal to $\mathcal V$ with respect to the metric $\kappa$.
Equivalently, $\mathcal H$ is the zero level-set
of the tangent bundle momentum mapping $\Phi_l$:
\begin{equation} \label{horizontal}
\mathcal H=\Phi_l^{-1}(0), \quad (\Phi_l(q,\dot q),\,\xi)
=\left(\frac{\partial{l}}{{\partial \dot q}},\,\xi_q\right)
=(\kappa_q \dot q,\,\xi_q), \quad \xi\in\g.
\end{equation}

The horizontal distribution $\mathcal H$ is $G$-invariant
and the quotient space $\mathcal H/G$
can be naturally identified with the tangent bundle $T(Q/G)$.
The restriction of the metric $\kappa$ to $\mathcal H$
define the \emph{submersion metric} $K$ on the reduced space $Q/G$ via
$$
K(X_1,X_2)_{\rho(q)}=\kappa(\bar X_1,\bar X_2)_q, \quad X_i\in T_{\rho(q)}(Q/G), \;
\bar X_i \in \mathcal H_q, \; X_i=d\rho(\bar X_i).
$$

Since
$(T^*Q)_0=\Phi^{-1}(0)=\bigcup_q \ann \mathcal V_q$, \
$\ann \mathcal V_q=\{p\in T^*_q Q\mid (p,\xi_q)=0,\,\xi\in\g\}$,
after passing to the quotient spaces $\mathcal H/G\approx T(Q/G)$
and $(T^*Q)_0/G\approx T^*(G/H)$ we see that the reduced Hamiltonian
is the Hamiltonian of the natural mechanical system $(Q/G,K,V)$.

\begin{thm} \label{LR}
The trajectories $q(t)$ of the $G$-invariant, natural mechanical system $(Q,\kappa,v)$
with velocities $q(t)$ that belong to $\mathcal H$ project to the trajectories $\rho(q(t))$
of the reduced natural mechanical system $(Q/G,K,V)$
\end{thm}

For Abelian groups this is the classical method of Routh
for eliminating cyclic coordinates \cite{Ro}.
The non-Abelian construction for the zero level-set of and for the other values
of the momentum mapping is given in \cite{AKN} and \cite{MRS}, respectively.
The above formulation is taken from \cite{AKN}.
In the case of geodesic flows the horizontal geodesic lines projects
to the geodesic lines of the submersion metric.

\section{Integrable Systems Related to Hamiltonian Actions}

\subsection{Collective Motions}
Consider the following two natural classes of functions on $M$.
Let ${\mathcal F}_1$ be the set of functions in $C^\infty(M)$
obtained by pulling-back the algebra $C^\infty(\mathfrak g^*)$ by
the moment map ${\mathcal F}_1=\Phi^*C^\infty(\mathfrak g^*)$. Let
${\mathcal F}_2$ be the set of $G$-invariant functions in
$C^\infty(M)$. The mapping $f\mapsto f\circ\Phi$ is a morphism of
Poisson structures:
\begin{equation} \label{2.4}
 \{f_1\circ\Phi,f_2\circ\Phi\}(x)=\{f_1,f_2\}_{\mathfrak g^*}(\eta), \quad \eta=\Phi(x),
\end{equation}
where $\{\cdot,\cdot\}_{\mathfrak g^*}$ is the Lie--Poisson bracket on $\mathfrak g^*$:
\begin{equation} \label{2.5}
\{f_1,f_2\}_{\mathfrak g^*}(\eta)=(\eta,[df_1(\eta),df_2(\eta)]), \quad
f_1,f_2: \mathfrak g^*\to \mathbb{R}.
\end{equation}
Thus, ${\mathcal F}_1$ is closed under the Poisson bracket.
Since $G$ acts in a Hamiltonian way, ${\mathcal F}_2$
is closed under the Poisson bracket as well.
The second essential fact is that $h\circ\Phi$ commute with any $G$-invariant function
(the Noether theorem).
In other words $\{{\mathcal F}_1,{\mathcal F}_2\}=0$.

\begin{assumption} [Separation of generic orbits by invariant functions]
Let a general orbit of the action have dimension $m$.
We shall suppose that
\begin{equation} \label{2.7}
\Span\{df(x),\; f\in{\mathcal F}_2\}=\ann (T_x(G\cdot x)),
\end{equation}
for general $x\in M$, $\dim G\cdot x=m$.
Whence $\ddim{\mathcal F}_2=2n-m$.
By $\reg{\mathcal F}_2$ denote the open dense set where (\ref{2.7}) holds.
\end{assumption}

Assumption 4.1 holds for any proper group action because all orbits
are separated by invariant functions.
Moreover, for the action of a compact group $G$ the algebra ${\mathcal F}_2$
is generated by a finite number of functions.
To be more precise, let the $G$ action have a finite number of orbit types.
Then there exist functions $f_1,\dots,f_r\in{\mathcal F}_2$,
such that every function $f\in{\mathcal F}_2$ is of the form: $f=F(f_1,\dots,f_r)$.
This theorem was proved by Schwarz \cite{[29]}.
If $M$ is compact then $M$ has a finite number of orbit types.
Furthermore, Mann proved that if $M$ is an orientable manifold
whose homology groups $H_i(M,\mathbb{Z})$ are finitely generated,
then the number of orbit types of any action of a compact Lie group on $M$ is finite.
A review of results concerning invariant functions of $G$ actions can be found in \cite{[27]}.

The following theorem, although it is a reformulation of some well known facts
about the momentum mapping (e.g., see \cite{GS}),
is fundamental in the considerations below.

Let $\mathcal A\subset C^\infty (\mathfrak g)$ be a Lie subalgebra
and $\Phi^*\mathcal A=\{h\circ\Phi,\;h\in \mathcal A\}$
the pull-back of $\mathcal A$ by the momentum mapping.
Then we have:

\begin{thm} [Bolsinov and Jovanovi\'c \cite{BJ2}] \label{Bol_Jov}
Let a connected Lie group $G$
act on $2n$-dimensional connected symplectic manifold $(M,\omega)$.
Suppose the action is Hamiltonian and Assumption 4.1 holds.
Then:

\emph{(i)} The algebra of functions ${\mathcal F}_1+{\mathcal F}_2$ is complete:
$$
\ddim({\mathcal F}_1+{\mathcal F}_2)+\dind({\mathcal F}_1+{\mathcal F}_2)=\dim M.
$$

\emph{(ii)} $\Phi^*{\mathcal A}+{\mathcal F}_2$ is a complete algebra on $M$
if and only if $\mathcal A$ is a complete algebra on a generic adjoint orbit
${\mathcal O}(\mu)\subset \Phi(M)$.

\emph{(iii)} If $\mathcal B$ is complete (commutative) subalgebra of $\mathcal{F}_2$
and $\mathcal A$ is complete (commutative) algebra on the orbit ${\mathcal O}(\mu)$,
for generic $\mu\in\Phi(M)$ then $\Phi^*{\mathcal A}+{\mathcal B}$
is complete (commutative) algebra on $M$.
\end{thm}

\begin{proof}
It is enough to prove item (i).
We shall need a well known fact that $\ker d\Phi(x)$ is symplectically orthogonal
to the tangent space at $x$ to the orbit of $x$ (see \cite{GS, LM}):
\begin{equation} \label{2.8}
\ker d\Phi(x)=(T_x(G\cdot x))^\omega.
\end{equation}

Let $U=\reg({\mathcal F}_1+{\mathcal F}_2)=\reg {\mathcal F}_2$
be the open dense set in $M$, such that for $x\in U$ the moment map (\ref{moment_map})
has maximal rank (by (\ref{2.8}) it is equivalent to the fact
that the orbit $G\cdot x$ has maximal dimension) and that (\ref{2.7}) holds.
Let $x$ belong to $\reg({\mathcal F}_1+{\mathcal F}_2)$.
Consider the tangent spaces to the level sets of integrals $\FF_1$ and $\FF_2$:
\begin{eqnarray*}
&D_1=D_1(x)=\Span\{X \mid (df_x,X)=0,\, f\in {\mathcal F}_1\}\subset T_x M \\
&D_2=D_2(x)=\Span\{X \mid (df_x,X)=0,\, f\in {\mathcal F}_2\}\subset T_x M
\end{eqnarray*}
We shall prove that $D_1\cap D_2$ is isotropic: $D_1\cap D_2\subset(D_1\cap D_2)^\omega$
which is equivalent to the completeness of the algebra ${\mathcal F}_1+{\mathcal F}_2$
(see remark \ref{SNI}).
For $x\in\reg({\mathcal F}_1+{\mathcal F}_2)$
we have $D_1=\ker d\Phi(x)$ and $D_2=T_x(G\cdot x)$.
Therefore by (\ref{2.8})
\[
(D_1\cap D_2)^\omega=D_1^\omega+D_2^\omega=D_2+D_1\supset D_1\cap D_2.\qedhere
\]
\end{proof}

Following Arnold \cite{Ar}, define $\ad^*$ by
$$
(\ad^*_\omega \mu,\xi)=(\mu,\ad_\omega \xi)=(\mu,[\omega,\xi]),
\quad \mu\in\g^*, \, \omega,\xi\in \g.
$$

Let $f_h=h\circ\Phi\in{\mathcal F}_1$ and let $\mu_0=\Phi(x_0)$. If
\begin{eqnarray} \label{invariant}
\ad^*_{dh(\mu_0)}\mu_0=0
&\Longleftrightarrow & dh(\mu_0)\in\ann T_{\mu_0}{\mathcal O}(\mu_0)
\nonumber\\
&\Longleftrightarrow & \{f,h\}_{\g^*}(\mu_0)=0, \, f\in C^{\infty}(\g^*),
\end{eqnarray}
then by
(\ref{2.4}) and (\ref{2.5}) we get that $f_h$ commutes
with all functions from ${\mathcal F}_2$ at $x_0$.

By the proof of Theorem 4.1, for $x_0\in\reg({\mathcal F}_1+{\mathcal F}_2)$,
the condition (\ref{invariant}) implies
that $df_h(x_0)$ belongs to the $\Span\{df(x_0),\;f\in{\mathcal F}_2\}$.
Thus we have:
\begin{eqnarray*}
\ddim({\mathcal F}_1+{\mathcal F}_2)
&=&\dim M-\dim G\cdot x+\dim {\mathcal O}(\mu) \\
&=&\dim M+\dim G_x-\dim G_\mu \\
\dind({\mathcal F}_1+{\mathcal F}_2)
&=&\dim G_\mu-\dim G_x,
\end{eqnarray*}
for general $x\in M$, $\mu=\Phi(x)$
($G_\mu$ and $G_x$ denotes the isotropy groups of $ G$ action at $\mu$ and $x$).

Note that if $h:\mathfrak g^*\to\mathbb{R}$ is an invariant of the co-adjoint action,
then (\ref{invariant}) hold on $\g^*$
and $h\circ\Phi$ will be a $G$-invariant function on $M$.

\begin{rem}
Similar construction can be performed on dual pairs (see Pana\-syuk \cite{P}).
Recall, the dual pair is given by two Poisson mappings
\[
N_1\;\overset{\Phi_1}{\longleftarrow}\;M\;\overset{\Phi_2}{\longrightarrow}\;N_2
\]
such that all level-sets of $\Phi_1$ and $\Phi_2$ are regular
and symplectically orthogonal in the intersection points \cite{Wa}.
Then the algebra $\FF_1+\FF_2$ is complete on $M$,
where $\FF_i=\Phi_i^*(C^\infty(N_i))$, $i=1,2$.
\end{rem}

\begin{dfn} \cite{GS}
The Hamiltonian $H: M\to \mathbb{R}$ is said to be \emph{collective}
if $H$ is of the form $H=h\circ\Phi$.
\end{dfn}

The solutions of the Hamiltonian equation $\dot x=X_{h\circ\Phi}$ are of the form
$$
x(t)=g(t)\cdot x_0,
$$
where $g(t)$ are solutions of the kinematic equation $\dot g=dh(\mu(t))\cdot g(t)$
and $\mu(t)$ are solutions of the Euler equations (\ref{Euler}) (see \cite{GS}).

For example, if $M=T^*G$, the geodesic flow of the \emph{right-invariant} metric on $G$
is described by the collective Hamiltonian with respect to the natural \emph{left} $G$-action.

\begin{cor}
Suppose that assumption 4.1 holds.
Let $h:\mathfrak g^*\to\mathbb{R}$ be a Hamiltonian function such that the Euler equations:
\begin{equation} \label{Euler}
\dot \mu=-\ad^*_{dh(\mu)}\mu \; \Longleftrightarrow \;
\dot f=\{f,h\}_{\mathfrak g^*}, \quad f\in C^{\infty}(\g^*),
\end{equation}
are completely integrable on general co-adjoint orbits ${\mathcal O}(\mu)\subset\Phi(M)$
with a set of Lie--Poisson commuting integrals $f_i$, $i=1,\dots,\frac12\dim\mathcal O(\mu)$.
Then the Hamiltonian equations on $M$ with the collective Hamiltonian function $h\circ\Phi$
are completely integrable (in the non-commutative sense).
The complete set of first integrals is
$$
\{f_i\circ\Phi\mid i=1,\dots,\tfrac12\dim\mathcal O(\mu)\}+{\mathcal F}_2.
$$
\end{cor}

The above theorem for the \emph{multiplicity-free} actions
is given by Guillemin and Sternberg (see \cite{GS1, GS2, GS}).

If $G$ is a compact group, then the connected components of regular invariant submanifolds
are isotropic tori of dimension $\frac12(\dim G+\dim G_\mu)-\dim G_x$.

\subsection{Geodesic Flows on Homogeneous Spaces of Compact Lie Groups}
The classical example is the geodesic flow of a left-invariant metric on the Lie group $SO(3)$.
The geodesic flow of such a metric describes the motion of a rigid body
about a fixed point under its own inertia.
This problem was solved by Euler (e.g., see \cite{Ar, AKN}.

In general, the geodesic flow of a left-invariant metrics on a Lie group $G$
after $G$-reduction reduces to the $\mathfrak g^*=(T^*G)/G$.
The reduced bracket of the canonical Poisson bracket on $T^*G$
are the Lie--Poisson bracket (\ref{2.5}) (multiplied by $-1$) \cite{Ar}.
Therefore the reduced system is described by Euler equations (\ref{Euler})
(multiplied by $-1$), where now $h$ is the reduced Hamiltonian.

A multidimensional generalization of the Euler case to $so(n)$
has been suggested by Manakov \cite{Ma}.
Using his idea, Mishchenko and Fomenko proposed the argument shift method (see below)
and constructed integrable examples of Euler equations for all compact groups \cite{MF1}
and proved the integrability of the original geodesic flows \cite{MF2, MF3}.
There are many other important constructions on various Lie algebras
(e.g., see \cite{TF, Bo, RS}).

Thimm \cite{Th} and Mishchenko \cite{Mis} proved the complete integrability
of geodesic flows of normal metrics on compact symmetric spaces $G/H$.
By a \emph{normal} $G$-invariant Riemannian metric
on the homogeneous space $G/H$ of a compact group $G$,
we mean the submersion metric induced from a bi-invatiant metric on $G$.

These results are generalized to the class of homogeneous spaces $G/H$
on which all $G$-invariant Hamiltonian systems are integrable
by means of Noether integrals (see \cite{GS1, GS2, Mik3}).
In those cases $(G,H)$ is a spherical pair and $G/H$
is a weekly symmetric space (see \cite{Vi} and references therein).
Examples of homogeneous, non (weekly) symmetric spaces $G/H$
with integrable geodesic flows are given in \cite{Th, PS, MS}.
It appears that many of those examples can be considered together
within the framework of non-commutative integrability.

Let $G$ be a compact connected Lie group with the Lie algebra $\mathfrak g=T_e G$.
Let us fix some bi-invariant metric $ds^2_0$ on $G$,
i.e., $\Ad_G$-invariant scalar product $\langle\cdot,\cdot\rangle$ on $\mathfrak g$.
We can identify $\mathfrak g^*$ and $\mathfrak g$ by $\langle\cdot,\cdot\rangle$.

Consider an arbitrary homogeneous space $G/H$ of the Lie group $G$.
The metric $ds^2_0$ induces so called \emph{normal metric}
(or \emph{standard metric}) on $G/H$.
We shall denote the normal metric also by $ds^2_0$.
By the use of $ds^2_0$ we identify $T^*G\cong TG$ and $T^*(G/H)\cong T(G/H)$.
Let $\mathfrak h$ be the Lie algebra of $H$
and $\mathfrak g=\mathfrak h+\mathfrak v$ the orthogonal decomposition.
Then $\mathfrak v$ can be naturally identified with $T_{\rho(e)}(G/H)$
and $T^*_{\rho(e)}(G/H)$, where $\rho:G\to G/H$ is the canonical projection.

The momentum mapping $\Phi:T^*(G/H)\to\mathfrak g$
of the natural $G$-action on $T^*(G/H)$ is given by $\Phi(g\cdot\eta)=\Ad_g(\eta)$.
where $\eta\in\mathfrak v$, \;$g\cdot\eta\in T^*_{\rho(g)}(G/H)$.

The Hamiltonian function of the geodesic flow of the normal metric $ds^2_0$
is simply given by
\begin{equation} \label{normal}
h_0=\frac 12 \langle \Phi,\Phi \rangle.
\end{equation}

Since $\langle\xi,\xi\rangle$ is $\Ad_G$-invariant,
$h_0$ Poisson commute with $\mathcal{F}_1$ and $\mathcal{F}_2$.
Note that for $\mathcal{F}_1$ and $\mathcal{F}_2$
we can take analytic functions, polynomial in momenta.
Thus we get

\begin{thm} \cite{BJ1}
Let $G$ be a compact Lie group.
The geodesic flows of the normal metrics on the homogeneous spaces $G/H$
are completely integrable in the non-commutative sense
by means of analytic functions, polynomial in momenta.
\end{thm}

Let $\mathbb{R}[\mathfrak g]^G$ be the algebra
of $\Ad_G$-invariant polynomials on $\mathfrak g$.
The polynomials
$$
{\mathcal A}_a=\{p^\lambda_a=p(\cdot+\lambda a)\mid
\lambda\in\mathbb{R},\;p\in\mathbb{R}[\mathfrak g]^G\}
$$
obtained from the invariants by shifting the argument are all in involution \cite{MF1}.
Furthermore, for \emph{every} adjoint orbit in $\mathfrak g$, one can find $a\in \mathfrak g$,
such that ${A}_a$ is a complete involutive set of functions on this orbit.
For regular orbits it is proved in \cite{MF1}.
For singular orbits there are several different proofs, see \cite{Mik1, Br3, Bo}.
(Note that $\Phi(T^*(G/H))$ can be often a subset of singular set in $\mathfrak g^*$.)

According item (ii) of theorem \ref{Bol_Jov},
$\Phi^*\mathcal A_a+\mathcal F_2$ is a complete algebra on $T^*(G/H)$.
Let $h_C(\xi)=\frac12\langle C(\xi),\xi\rangle$ be a quadratic positive definite polynomial
in $\mathcal A_a$ (operators $C:\g\to\g$ are described in \cite{MF1}).
Then $h_C\circ\Phi$ is the Hamiltonian of the geodesic flow of a certain metric
that we shall denote by $ds_{C}^2$.
The metric $ds^2_{C}$ has the following nice geometrical meaning.
This is the submersion metric of the right-invariant Riemannian metric on $G$
whose Hamiltonian function is obtained from $h_C(\xi)$ by right translations.
Thus some integrable (in non-commutative sense) deformations of the normal metric
always exist \cite{BJ1, BJ2}.
This construction for symmetric spaces is done by Brailov \cite{Br2, Br3}.

Whence, to construct a complete commutative algebra of analytic integrals
of the geodesic flow of the normal metric we need to solve the following nontrivial problem:
to find a complete commutative subalgebra $\mathcal B$
of $G$-invariant functions on $T^*(G/H)$.

For symmetric and weekly symmetric spaces (spherical pairs) the algebra $\mathcal{F}_2$
is already commutative and for the integrability we need only Noether's integrals $\FF_1$.
Spherical pairs for $G$ simple and semisimple
are classified in \cite{Kr} and \cite{Mik3}, respectively.

The examples of Thimm \cite{Th} ($(SO(n),SO(n-2))$) and Paternain and Spatzier \cite{PS}
($(SU(3),\mathbb{T}^2)$) are \emph{almost spherical pairs},
that are the homogeneous spaces where for a complete commutative set of functions on $T^*(G/H)$
we can take an arbitrary complete commutative subalgebra $\mathcal A\subset\mathcal F_1$
and one $G$-invariant function functionally independent of $\mathcal{F}_1$ \cite{MS, Mik4}.
Almost spherical pairs are classified, for $G$ compact and semisimple in \cite{Pa, MS}.

There are two natural methods for constructing commutative families of $G$-invariant functions,
namely the shift-argument method and chain of subalgebras method, see \cite{BJ1, BJ3}.
In many examples (Stiefel manifolds, flag manifolds,
orbits of the adjoint actions of compact Lie groups etc.)
those methods lead to complete commutative algebras
(see Bolsinov and Jovanovi\'c \cite{BJ1, BJ3}, Buldaeva \cite{Bul},
Mykytyuk and Panasyuk \cite{MP}).

\begin{rem}
We also note that integrable potential systems on symmetric spaces and coadjoint orbits
can be found, e.g., in \cite{Sa} and \cite{BJ5}, respectively.
Then one can easily construct, by means of the classical Maupertuis principle,
corresponding integrable geodesic flows that are not $G$ invariant (e.g., see \cite{BF}).
\end{rem}

\subsection{ Magnetic Geodesic Flows}
Similar construction is used for the magnetic geodesic flows of the normal metrics
on a class of homogeneous spaces,
in particular (co)adjoint orbits of compact Lie groups \cite{Ef1, Ef2, BJ4, BJ5}.

The geodesic flow on $(Q,\kappa)$ can be interpreted as the inertial motion
of a particle on $Q$ with the kinetic energy given by (\ref{ham}) (we take $v\equiv 0$).
The motion of the particle under the influence of the additional magnetic field
given by a closed 2-form
$$
\Omega=\sum_{1 \leq i<j \leq n} F_{ij}(q) dq^i \wedge dq^j,
$$
is described by the following equations:
\begin{equation} \label{magnetic_flow}
\frac{dq^i}{dt}=\frac{\partial h}{\partial p_i}, \qquad
\frac{dp_i}{dt}=-\frac{\partial h}{\partial q^i}+\sum_{j=1}^{n}F_{ij}
\frac{\partial h}{\partial p_j}.
\end{equation}

The equations (\ref{magnetic_flow}) are Hamiltonian
with respect to the ``twisted" symplectic form $\omega+\sigma^*\Omega$,
where $\sigma: T^*Q\to Q$ is the natural projection.
Namely, the new Poisson bracket is given by
\begin{equation} \label{magnetic_bracket}
\{f,g\}_\Omega=\{f,g\}
+\sum_{i,j=1}^{n} F_{ij} \frac{\partial f}{\partial p_i}\frac{\partial g}{\partial p_j},
\end{equation}
and the Hamiltonian equations $\dot f=\{f,h\}_\Omega$ read (\ref{magnetic_flow}).
Here $\{\cdot,\cdot\}$ are canonical Poisson bracket (\ref{CPB}).

The flow (\ref{magnetic_flow}) is called \emph{magnetic geodesic flow}
on the Riemannian manifold $(Q,\kappa)$ with respect to the magnetic field $\Omega$.
For simplicity, we shall refer to (\ref{magnetic_bracket})
as a \emph{magnetic Poisson bracket}
and to $(T^*Q,\omega+\sigma^*\Omega)$ as a \emph{magnetic cotangent bundle}.

We introduce a class of homogeneous spaces $G/H$
having a natural construction of the magnetic term, consisting of pairs $(G,H)$,
where $H$ have one-point adjoint orbits.
Let $a\in\mathfrak h$ be the $H$-adjoint invariant.
Then $H$ is a subgroup of the $G$-adjoint isotropy group $G_a$ of $a$.
The adjoint orbit $\mathcal O(a)$ through $a$
carries the Kirillov--Konstant symplectic form $\Omega_{KK}$ (e.g., see \cite{GS}).
Then we have canonical submersion of homogeneous spaces
$
\tau: G/H \to G/G_a\cong \mathcal O(a)
$
and the closed two-form $\Omega=\tau^*\Omega_{KK}$
gives us the required magnetic field on $G/H$.

The form $\Omega$ is $G$-invariant.
From the definition of $\Omega_{KK}$ one can easily prove
that at the point $\rho(e)$, $\Omega$ is given by
\begin{equation} \label{magnetic_form}
\Omega(\xi_1,\xi_2)|_{\rho(e)}=-\langle a,[\xi_1,\xi_2]\rangle, \quad
\xi_1,\xi_2\in\mathfrak v \cong T_{\rho(e)}(G/H)
\end{equation}

\begin{lem}
The momentum mapping $\Phi_\epsilon:T^*(G/H)\to\mathfrak g$
of the natural $G$-action on $T^*(G/H)$
with respect to the symplectic form $\omega+\epsilon \sigma^*\Omega$ is given by
$$
\Phi_\epsilon(g\cdot\eta)=\Phi_0(g\cdot\eta)+\epsilon \Ad_g(a),
$$
where $\eta\in\mathfrak v$, \,$g\cdot\eta \in T^*_{\pi(g)}(G/H)$.
\end{lem}

Now, let $\mathcal F_1^\epsilon$ be the algebra of all analytic,
polynomial in momenta, functions of the form
$
\mathcal F_1^\epsilon=\{p\circ\Phi_\epsilon, \, p\in \R[\mathfrak g]\}
$
and $\mathcal F_2$ be the algebra of all analytic, polynomial in momenta,
$G$-invariant functions on $T^*(G/H)$.  Then
$
\{\mathcal F_1^\epsilon,\mathcal F_2\}_\epsilon=0,
$
where $\{\cdot,\cdot\}_\epsilon$ are magnetic Poisson bracket
with respect to $\omega+\epsilon\sigma^*\Omega$.

Consider the Hamiltonian
$h_\epsilon=\frac12\langle\Phi_\epsilon,\Phi_\epsilon \rangle\in \mathcal F_1^\epsilon$.
We have
$$
h_\epsilon (g\cdot\eta)
=\frac12\langle\Ad_g\eta,\Ad_g\eta\rangle+\epsilon\langle \Ad_g \eta,\Ad_ga\rangle
+\frac{\epsilon^2}{2}\langle \Ad_g a,\Ad_g a \rangle=h_0(g\cdot\eta)+\operatorname{const},
$$
where we used that $\eta\in\mathfrak v$ is orthogonal to $a\in\mathfrak h$
and $h_0$ is the Hamiltonian of the normal metric (\ref{normal}).
Thus, the Hamiltonian flows of $h_0$ and $h_\epsilon$ coincides.
Since $h_{\epsilon}$ belongs to $\mathcal F_1^\epsilon$ its commutes with $\mathcal F_2$.
On the other side, as a composition of the momentum mapping with an invariant polynomial,
the function $h_\epsilon$ is also $G$-invariant and commutes with $\mathcal F_1^\epsilon$.

\begin{thm} [Bolsinov and Jovanovi\'c \cite{BJ5}]
The magnetic geodesic flow of the normal metrics $ds^2_0$ on the homogeneous space $G/H$
with respect to the $G$-invariant closed $2$-form $\epsilon\Omega$ given at $\rho(e)$
by \eqref{magnetic_form} is completely integrable in the non-commutative sense.
The complete algebra of first integrals is $\mathcal F_1^\epsilon+\mathcal F_2$.
\end{thm}

\section{Reductions and Integrability}

\subsection{Reconstruction}
Suppose the reduced flow (\ref{reduced-flow}) is completely integrable
(in the non-commutative sense) with a complete algebra of first integrals $\FF$.
Then $\FF$ can be considered as a complete algebra
$$
\mathcal B=\pi^*\FF
$$
of $\FF_2=C^\infty(M)^G$.
According Theorem 4.1 we get the following theorem, explicitly given in \cite{Zu}.

\begin{thm} [Zung \cite{Zu}] \label{M-F-Z}
The integrability of the reduced flow \eqref{reduced-flow}
implies the integrability of the original system \eqref{hamiltonian}.
\end{thm}

In the case of geodesic flows of left-invariant metric,
Theorem \ref{M-F-Z} is proved by Mishchenko and Fomenko \cite{MF2, MF3}.

\subsection{Poisson Reduction}
We have the following simple general observation (see e.g.~\cite{Jo1}).

\begin{lem}
Suppose a connected Lie group $G$ acts effectively on a symplectic manifold $(M,\omega)$.
Let $h$ be a $G$-invariant Hamiltonian function.

\emph{(i)} If the Hamiltonian system \eqref{hamiltonian} is
completely integrable by means of $G$-invariant first integrals,
then $G$ is commutative. Moreover, if the action is free, the
reduced Hamiltonian system on $N_\eta$ is completely integrable
for a generic value $\eta$ of the moment map, i.e., the reduced
flow \eqref{reduced-flow} on the Poisson manifold $M/G$ is
completely integrable.

\emph{(ii)} If the system \eqref{hamiltonian} is $\T^n$-dense completely integrable,
then the Lie group $G$ is commutative.
\end{lem}

\begin{proof}
(i) Let $\FF$ be a complete $G$-invariant algebra of integrals
and let $D_x$ and $W_x$ be the subspaces of $T_x M$ defined in Remark 2.1.
Since all functions from $\FF$ are $G$-invariant
we have that $T_x(G\cdot x)\subset D_x$ is isotropic, for generic $x\in M$.
The Hamiltonian vector fields of the functions (\ref{one})
generate one-parameter groups of symplectomorphisms.
Therefore $\omega_x(X_{\phi_{\xi_1}}(x),X_{\phi_{\xi_2}}(x))$ vanishes
in an open dense set of $M$.
Thus, according to (\ref{com})
$$
\phi_{[\xi_1,\xi_2]}=\{\phi_{\xi_1},\phi_{\xi_2}\}=0
$$
for every $\xi_1,\xi_2 \in \g$.
Since the action is effective, we get that $G$ is commutative.

Let the Abelian group $G$ act freely on $M$.
Take $x\in M$ such that $\FF$ is complete at $x$.
Since all functions from $\FF$ are $G$-invariant,
the coisotropic space $W_x=\{X_f(x)\mid f\in\FF\}$
belongs to $T_x \Phi^{-1}(\eta)$, $\eta=\Phi(x)$.

Let $\FF_\eta=\{F_\eta, f\in\FF\}$ be the induced algebra of integrals on $N_\eta$.
From the definition of the reduced symplectic structure, one can easily see that the space
$$
W_x^\eta=\{X_{F_\eta}(\pi_\eta(x)) \mid F\in\FF_\eta\}=d\pi_\eta(W_x)
$$
is also coisotropic, i.e., the induced algebra $\FF_\eta$ is complete at $\pi_\eta(x)$.

\smallskip
(ii) If the system (\ref{hamiltonian}) is $\T^n$-dense completely integrable,
then in every toroidal domain the functions $\phi_\xi(x)$, $\xi\in\g$
depend only on the action variables and so
$\{\phi_{\xi_1},\phi_{\xi_2}\}=\phi_{[\xi_1,\xi_2]}$ vanishes on $\reg M$.
Therefore $[\g,\g]=0$ and the Lie group $G$ is commutative.
\end{proof}

From Lemma 5.1, the condition on a $G$-invariant algebra $\FF$ on $(M,\omega)$
to be complete forces $G$ to be abelian, which is too restrictive.
However, we can consider the case when we have integrability of (\ref{hamiltonian})
with non-invariant integrals.
Also, note that by item (ii) the existence of the non-Abelian group of symmetries
is closely related to the non-commutative integrability of the Hamiltonian flow.

\begin{dfn} \cite{Zu}
We shall say that the Hamiltonian system (\ref{hamiltonian})
is \emph{completely integrable by means of an algebra of all integrals}
if the family of all integrals of the system form a complete algebra.
\end{dfn}

It is clear that when $\FF$ is a complete algebra of integrals, $\dind\FF=r$,
the connected components of regular invariant manifolds are compact
and the system is $\T^r$-dense within $\reg M$,
then the system is integrable by means of an algebra of all integrals.

\begin{thm} [Zung \cite{Zu}]
If $G$ is compact and the Hamiltonian system \eqref{hamiltonian}
is completely integrable by means of an algebra of all integrals,
then the reduced system \eqref{reduced-flow} on $M/G$
will be integrable by means of an algebra of all integrals.
\end{thm}

\begin{proof}
For the completeness in the exposition, we shall present the proof given in \cite{Zu}
slightly modified and adopted for our notation.

Let $\FF_h$ be the algebra of all integrals of the system
(\ref{hamiltonian}). Without loss of generality it can be assumed
that there exist functions $f_1,\dots,f_l,h_1,\dots,h_r\in \FF_h$
($l=\dind\FF_h$, $r=\dind\FF_h$) such that $\{\FF_h,h_i\}=0$,
$i=1,\dots,r$ and $df_1\wedge\dots\wedge df_l\neq 0$ and
$dh_1\wedge \dots \wedge dh_r\neq 0$ at a generic point in $M$.

As in Section 3, by capital letters we denote the functions on $M/G$
and by small letters corresponding $G$-invariant functions on $M$ ($f=\pi^*(F)$,
where $\pi: M\to M/G$ is the natural projection).

Let $s$ be the generic dimension of the intersection
of a common level set of $f_i$, $i=1,\dots,l$ with an orbit of $G$ in $M$.
The set of all integrals $\FF_H$ of the flow of the reduced Hamiltonian $H$
can be obtained from the functions in $\FF_h$ by averaging with respect to the $G$-action.
Therefore
\begin{equation} \label{Z1}
l-\ddim \FF_H=\ddim\FF_h-\ddim\FF_H \leq \dim G-s.
\end{equation}

An important observation is that the functions $h_i$ are $G$-invariant.
Since $h$ is $G$ invariant, the functions (\ref{one}) are first integrals,
i.e., $\phi_\xi\in \FF_h$, \;$\xi\in\g$.
This implies $\{\phi_\xi, h_i\}=0$, $i=1,\dots,r$, \;$\xi\in\g$,
which means that $h_i$ are $G$-invariant.  Whence
\begin{multline} \label{Z2}
r-\dim\Span\{X_{H_1},\dots,X_{H_r}\}|_{\pi(x)}\\
=\dim\Span\{X_{h_1},\dots,X_{h_r}\}|_{x}-\dim\Span\{X_{H_1},\dots,X_{H_r}\}|_{\pi(x)}=s,
\end{multline}
for generic $x\in M$.
Here $X_{h_i}$ and $X_{H_i}$ are Hamiltonian vector fields on $M$ and $M/G$, respectively.

Further, the functions $H_1,\dots,H_r$ belongs to the kernel of $\FF_H$.
Indeed if $F\in \FF_H$, then $f=\pi^*(F)\in \FF_h$ and
$
\{h_i,f\}=\{H_i,F\}^G=0.
$

Taking into account $\rank G$-Casimir functions of the Poisson bracket $\{\cdot,\cdot\}^G$,
from (\ref{Z2}) we get
\begin{equation} \label{Z3}
\dind\FF_H \geq r-s+\rank G.
\end{equation}

Finally, from (\ref{Z1}) and (\ref{Z3}) it follows
$$
\ddim \FF_H+\dind\FF_H \geq l+s-\dim G+r-s+\rank G=\dim M/G+\corank \{\cdot,\cdot\}^G
$$
implying the completeness of $\FF_H$.
\end{proof}

\subsection{Symplectic Reduction}
Suppose we are given a $G$-invariant natural mechanical system $(Q,\kappa,v)$
with a completely integrable flow.
The natural question arises: will the reduced natural system on $Q/G$ be integrable?

This approach is used in the construction of new manifolds
with completely integrable geodesic flows starting from the integrable geodesic flows
with symmetries.
Paternain and Spatzier proved that
if the manifold $Q$ has geodesic flow integrable by means of $S^1$-invariant integrals
and if $N$ is a surface of revolution,
then the submersion geodesic flow on $Q\times_{S^1} N=(Q\times N)/S^1$
will be completely integrable \cite{PS}.
Combining submersions and Thimm's method (see \cite{Th}), Paternain and Spatzier \cite{PS}
and Bazaikin \cite{Baz} proved integrability of geodesic flows
on certain interesting bi-quotients of Lie groups.
We shall explain these submersion examples in the framework of a general construction.

Note that, if $G$ is not Abelian,
$T^*(Q/G)$ is the singular symplectic leaf in $(T^*Q)/G$
so we can not directly apply Theorem 5.1.

Let $G$ be a compact connected Lie group.
Suppose we are given an integrable $G$-invariant Hamiltonian system $\dot x=X_H(x)$
with compact iso-energy levels $M_h=H^{-1}(h)$.
Then $M$ is foliated by invariant tori in an open dense set that we shall denote by $\reg M$.
Suppose $0$ is the regular value of the momentum mapping.

\begin{thm}
If $\reg M$ intersects the submanifold $\Phi^{-1}(0)$ in a dense set,
then the reduced Hamiltonian system \eqref{reduced-eta} on $(N_0,\omega_0)$
will be completely integrable.
\end{thm}

\begin{proof}
Take some toroidal domain $\OO=\T^r\{\varphi\}\times
B_\sigma\{I,p,q\} \subset \reg M$ such that
$\OO_1=\OO\cap\Phi^{-1}(0)\neq\emptyset$ and that the Hamiltonian
flow (\ref{hamiltonian}) is $\T^r$-dense in $\OO$. Consider the
$G$-invariant sets
$
\UU_1=\UU\cap \Phi^{-1}(0), \;
\UU=G\cdot \OO=\{g\cdot x\mid x\in \OO,\; g\in G\}.
$

Note that the functions (\ref{one}) are first integrals of the Hamiltonian flow
and so do not depend on $\varphi$ in $\OO$.
In other words, the action of $G$ preserves the foliation
by $r$-dimensional invariant isotropic tori of the sets $\UU$ and $\UU_1$ as well.

The foliation $\mathcal T$ of $\UU_1$ by tori induces a foliation $\mathcal L$ of $\UU_1$
by $(\dim G+r-s)$-dimensional $G$-invariant submanifolds, with tangent spaces of the form
$$
T_x\mathcal L=T_x(G\cdot x)+T_x(\mathcal T)\subset T_x\Phi^{-1}(0),
$$
where $s=\dim S_x$, \;
$
S_x=T_x(G\cdot x) \cap T_x(\mathcal T), \; x\in \UU_1
$
(for $x\in \OO_1$ we have that
$$
T_x\mathcal L=\{X_{\phi_a}(x)\mid X_{I_i}(x)\}, \quad
S_x=\{X_{\phi_a}(x)\} \cap \{X_{I_i}(x)\}
$$
and the foliation $\mathcal L|_{\OO_1}$ does not depend on
$\varphi$).

Let $\pi_0: \Phi^{-1}(0)\to \Phi^{-1}(0)/G=N_0$ be the natural projection
and let $\tilde \UU=\pi_0(\UU_1)\subset N_0$.
The foliation $\mathcal L$ induces a foliation $\tilde{\mathcal T}=\pi_0(\mathcal L)$
of $\tilde \UU$ by an $(r-s)$-dimensional invariant manifolds
of the geodesic flow of the submersion metric.
We shall see below that $\tilde{\mathcal T}$ is a foliation of $\tilde \UU$
by an invariant tori with respect to certain complete algebra of first integrals $\FF_0$.

The foliation $\mathcal L$ can be seen as the level sets of
$G$-invariant integrals $f_1,\dots,f_\rho$ on $\UU_1$, $\rho=\dim
M-\dim G+s-r$. Indeed, this is always true locally: for $\sigma$
small enough there are functions $f_i^{\OO_1}=f_i^{\OO_1}(I,p,q)$
on $\OO_1$ such that the tangent spaces $T_x\mathcal L$ (recall
that $\mathcal L|_{\OO_1}$ does not depend on $\varphi$) are given
by the equations
$$
T_x\mathcal L=\{\xi\in T_x \OO_1 \mid  df_i^{\OO_1}(x)(\xi)=0,\;i=1,\dots,\rho\}.
$$
Then functions $f_i$ are $G$-invariant extensions of $f_i^{\OO_1}$ to $\UU_1$,
$i=1,\dots,\rho$.

Let $\FF_0$ be the induced algebra of first integrals in $\tilde\UU$.
Since
$$
D_{\pi_0(x)}=\{\xi\in T_{\pi_0(x)} \tilde U \mid df_0(\pi_0(x))(\xi)=0,\; f_0\in\FF_0\}
=T_{\pi_0(x)}\tilde{\mathcal T}=d\pi_0(x)(T_x\mathcal L),
$$
$x\in \UU_1$ are isotropic spaces, by Remark 1 we get that $\FF_0$ is complete in $\tilde \UU$.
Note that
$
\ddim\FF_0=\dim N_0+s-r, \; \dind\FF_0=r-s.
$

Now fill up $N_0$ with countably many disjoint torodial domains
$$
\tilde{\OO}_\alpha
=\T^{r(\alpha)}\{\tilde\varphi\} \times B_{\sigma(\alpha)}\{\tilde I,\tilde p,\tilde q\}
$$
(with possible different dimensions of tori).
As in Theorem 2.2, in every domain $\tilde{\OO}_\alpha$,
one can construct complete involutive set of integrals that can be then ``glued"
in order to obtain a complete involutive set of integrals globally defined.

We have the following $m=\frac12\dim N_0$ commuting integrals in $\tilde{\OO}_\alpha$:
$$
h_1=\tilde{I}_1^2,\dots,h_r=\tilde{I}_r^2,\;\;
h_{r+1}=\tilde{p}_1^2+\tilde{q}_1^2,\dots,\;h_m=\tilde{p}_k^2+\tilde{q}_k^2.
$$

Let $g_\alpha:\mathbb{R}\to\mathbb{R}$ be a smooth nonnegative function
such that $g_\alpha(x)$ is equal to zero for $|x|>\sigma(\alpha)$,
$g_\alpha$ monotonically increases on $[-\sigma(\alpha),0]$
and monotonically decreases on $[0,\sigma(\alpha)]$.
Let $h_\alpha(y)=g_\alpha(h_1(y)+\dots+h_m(y))$.
This function can be extended by zero to the whole manifold $N_0$.
Then $f_i^\alpha=h_\alpha\cdot h_i$, $i=1,\dots,n$ will be commuting functions,
independent on an open dense subset of $\tilde{\OO}_\alpha$.
With a ``good" choice of $g_\alpha$'s,
a complete commutative set of smooth integrals is given by $f_i(y)=f_i^\alpha(y)$
for $y \in \tilde{\OO}_\alpha\subset\bigcup_\beta \tilde{\OO}_\beta$
and zero otherwise, $i=1,\dots,m$.
\end{proof}

Note that we do not suppose that integrals of the original system are $G$-invariant.
Also, the general construction used in the proof of the theorem
leads to smooth commuting integrals on $M_0$.

\begin{rem}
It is clear that a similar statement holds for an arbitrary value of the momentum mapping:
if $\reg M$ intersects the submanifold $\Phi^{-1}(\eta)$ in a dense set,
then the reduced Hamiltonian system on $(N_\eta,\omega_\eta)$ will be completely integrable.
Also one can get the following modification of Theorem 5.2:
Suppose we are given an integrable Hamiltonian system (\ref{hamiltonian})
invariant with respect to the action of a compact group $G$
and with compact iso-energy levels $M_h=H^{-1}(h)$.
Then the reduced system (\ref{reduced-flow}) on $M/G$ will be completely integrable.
\end{rem}

\begin{cor} \label{submersion} \cite{Jo1}
Let a compact connected Lie group $G$ act freely by isometries
on a compact Riemannian manifold $(Q,\kappa)$.
Suppose that the geodesic flow of $\kappa$ is completely integrable.
If $\reg T^*Q$ intersects the space of horizontal vectors $\mathcal H\cong\Phi^{-1}(0)$
in a dense set then the geodesic flow on $Q/G$
endowed with the submersion metric $K$ is completely integrable.
\end{cor}

\begin{exm}
Eschenburg constructed bi-quotients $M^7_{k,l,p,q}=SU(3)//T_{k,l,p,q}$
endowed with the submersion metrics $ds^2_{t,sub}$
with strictly positive sectional curvature.
Here $T_{k,l,p,q}\cong T^1 \subset T^2\times T^2$,
where $T^2$ is a maximal torus and $ds^2_t$ is a one-parameter family
of left-invariant metrics on $SU(3)$ (see \cite{E}).
One can prove that the geodesic flows of the metrics $ds^2_t$
are completely integrable and that we can apply Theorem \ref{submersion}
to get the integrability of the geodesic flows of the submersion metrics on $M^7_{k,l,p,q}$.
\end{exm}

Suppose we are given Hamiltonian $G$-actions
on two symplectic manifolds $(M_1,\omega_1)$ and $(M_2,\omega_2)$
with moment maps $\Phi_{M_1}$ and $\Phi_{M_2}$.
Then we have the natural diagonal action of $G$
on the product $(M_1\times M_2,\omega_1\oplus\omega_2)$, with moment map
\begin{equation} \label{c.1}
\Phi_{M_1\times M_2}=\Phi_{M_1}+\Phi_{M_2}.
\end{equation}

If $(Q_1,\kappa_1)$ and $(Q_2,\kappa_2)$ have integrable geodesic flows,
then $(Q_1\times Q_2, \kappa_1\oplus \kappa_2)$ also has integrable geodesic flow.
Using (\ref{c.1}), one can easily see that
if the $G$-actions on $Q_1$ and $Q_2$ are almost everywhere locally free,
and free on the product $Q_1 \times Q_2$,
then a generic horizontal vector of the submersion
$$
Q_1 \times Q_2 \to Q_1 \times_G Q_2=(Q_1\times Q_2)/G
$$
belongs to $\reg T^*(Q_1 \times Q_2)=\reg T^*Q_1 \times \reg T^*Q_2$.
Thus, with the above notation, we get

\begin{thm}
The geodesic flow on $Q_1 \times_G Q_2$, endowed with the submersion metric,
is completely integrable.
\end{thm}

\begin{exm}
Let a compact Lie group $G$ acts freely by isometries on $(Q,\kappa)$
and let $G_1$ be an arbitrary compact Lie group that contains $G$ as a subgroup.
Let $ds^2_1$ be some left-invariant Riemannian metric on $G_1$ with integrable geodesic flow.
Then $G$ acts in the natural way by isometries on $(G_1,ds^2_1)$.
Therefore, if the geodesic flow on $Q$ is completely integrable,
then the geodesic flows on $Q\times_G Q$ and $Q \times_G G_1$
endowed with the submersion metrics will be also completely integrable.
\end{exm}

\section{Partial Integrability}

In this section we study reductions of the Hamiltonian flows
restricted to their invariant submanifolds (see \cite{Jo2}).
Apparently, the lowering of order in Hamiltonian systems having invariant relations
was firstly studied by Levi-Civita (e.g., see \cite[ch.~X]{LC}).

\subsection{ Hess--Appel'rot System}
The classical example of the system having an invariant relation
is a celebrated Hess--Appel'rot case of a heavy rigid body motion \cite{He, App}.
Recall that the motion of a heavy rigid body around a fixed point,
in the moving frame, is represented by the Euler--Poisson equations
\begin{equation} \label{EP}
\frac{d}{dt}\vec m=\vec m\times \vec \omega+\mathfrak G\mathfrak M\,\vec\gamma\times\vec r,
\quad \frac{d}{dt}\vec\gamma=\vec \gamma\times \vec \omega, \quad \vec\omega=A\vec m,
\end{equation}
where $\vec\omega$ is the angular velocity, $\vec m$ the angular momentum,
$I=A^{-1}$ the inertial tensor, $\mathfrak M$ mass
and $\vec r$ the vector of the mass center of a rigid body;
$\vec\gamma$ is the direction of the homogeneous gravitational field
and $\mathfrak G$ is the gravitational constant.

The equations (\ref{EP}) always have three integrals, the energy,
geometric integral and the projection of angular momentum:
\begin{equation} \label{integrals11}
F_1=\frac12(\vec m,\vec \omega)+\mathfrak M\mathfrak G(\vec r,\vec\gamma),
\quad F_2=(\vec\gamma,\vec\gamma)=1, \quad F_3=(\vec m,\vec \gamma).
\end{equation}
For the integrability we need a forth integral.
There are three famous integrable cases: Euler, Lagrange and Kowalevskaya \cite{AKN, Go}.

Apart of these cases, there are various particular solutions (e.g., see \cite{Go}).
The celebrated is partially integrable Hess--Appel'rot case \cite{He, App}.
Under the conditions:
\begin{equation} \label{conditions}
r_2=0, \quad r_1\sqrt{a_3-a_2}\pm r_3\sqrt{a_2-a_1}=0, \quad A=\diag(a_1,a_2,a_3),
\end{equation}
$a_3>a_2>a_1>0$, system (\ref{EP}) has an invariant relation given by
\begin{equation}  \label{invariant_relation}
F_4=(\vec m,\vec r)=r_1 m_1+r_3 m_3=0.
\end{equation}

The system is integrable up to one quadrature:
the compact connected components of the regular invariant sets
\begin{equation} \label{tori}
F_1=c_1,\quad F_2=1,\quad F_3=c_3,\quad F_4=0
\end{equation}
are tori, but not with quasi-periodic dynamics.
The classical and algebro-geometric integration
can be found in \cite{Go} and \cite{DrGa1}, respectively.

There is a nice geometrical interpretation of the conditions (\ref{conditions})
given by Zhukovski: the intersection of the plane orthogonal to $\vec r$ with the ellipsoid
$$
(A\vec m,\vec m)=const
$$
is a circle \cite{Zh, BoMa}.
If instead of the moving base given by the main axes of the inertia,
we take the moving base $\vec f_1, \vec f_2, \vec f_3$
such that the mass center of a rigid body $\vec r$ is proportional to $\vec f_3$,
\[
\vec r=\rho \vec f_3, \quad \rho=\sqrt{r_1^2+r_3^2},
\]
then the inverse of the inertial operator reads
\[
A=\begin{pmatrix}
a_2 & 0 & a_{13} \\
0 & a_2 & a_{23} \\
a_{13} & a_{23} & a_{33}
\end{pmatrix}
\]
and the invariant relation (\ref{invariant_relation}) is simply given by
\begin{equation} \label{HA_m3}
m_3=0.
\end{equation}

Note that, considered on the whole phase space $T^*SO(3)$ of the rigid body motion,
the function $m_3$ is the momentum mapping of the right $SO(2)$-action-rotations
of the body around the line directed to the center of the mass.

Further, let $\vec R$ and $\vec\Gamma$ be the vectors of the mass center of a rigid body
and the direction of the gravitational field considered in the space frame.
The motion of $\vec R(t)$ is described by the spherical pendulum equations
(see \cite{Zh, BoMa})
\begin{equation} \label{spherical_pendulum}
\frac{d^2}{dt^2} \vec R=-a_2 \rho^2 \mathfrak G\mathfrak M\vec{\Gamma}+\mu \vec R
\end{equation}
Here the Lagrange multiplier $\mu$ is determined from the condition $(\vec R,\vec R)=\rho^2$.

Historical overview, with an application of Levi--Civita ideas
to the Hess--Appel'\-rot system and other classical rigid body problems
can be found in Borisov and Mamaev \cite{BoMa}.

Recently, an interesting construction of a $n$-dimensional variant
of the Hess--Appel'rot system is studied by Dragovi\'c and Gaji\'c \cite{DrGa2}.
Besides, by generalizing \emph{analytical and algebraic properties} of the classical system,
they introduced a class of systems with invariant relations,
so called \emph{Hess--Appel'rot type systems} with remarkable property:
there exist a pair of compatible Poisson structures,
such that the system is Hamiltonian with respect to the first structure
and invariant relations are Casimir functions with respect to the second structure
(for more details see \cite{DrGa2}).
On the other side, considering invariant relation (\ref{HA_m3}),
we are interested in the following problem:
suppose that $h$ is not a $G$-invariant function,
but $M_0$ is still an invariant manifold of the Hamiltonian system (\ref{hamiltonian}).
As a modification of the regular Marsden--Weinstein reduction,
we study (partial) reduction of the Hamiltonian system (\ref{hamiltonian})
from $M_0$ to $N_0$ and relationship between the integrability
of the reduced and nonreduced system.

\subsection{Invariant Relations and Reduction}
Let $G$ be a connected Lie group with a free proper Hamiltonian action
on a symplectic manifold $(M,\omega)$ with the momentum map (\ref{moment_map}).
Assume that $0$ is a regular value of $\Phi$.

Let $\xi_1,\dots,\xi_p$ be the base of $\g$.
Then the zero level set of the momentum mapping (\ref{moment_map}) is given by the equations
\begin{equation} \label{moments}
M_0: \qquad \phi_\alpha=(\Phi,\xi_\alpha)=0, \qquad \alpha=1,\dots,p.
\end{equation}

\begin{thm} \label{main}
\emph{(i)} Suppose that the restriction of $h$ to \eqref{moments} is a $G$-invariant function.
Then $M_0$ is an invariant manifold of the Hamiltonian system \eqref{hamiltonian}
and $X_{h}|_{M_0}$ projects to the Hamiltonian vector field $X_{H_0}$
\begin{equation} \label{projection1}
d\pi_0(X_h)|_x=X_{H_0}|_{y=\pi_0(x)},
\end{equation}
where $H_0$ is the induced function on $N_0$ defined by \eqref{induced_ham}.

\emph{(ii)} The inverse statement also holds: if \eqref{moments}
is an invariant submanifold of the Hamiltonian system
\eqref{hamiltonian}, then the restriction of $h$ to $M_0$ is a
$G$-invariant function and $X_{h}|_{M_0}$ projects to the
Hamiltonian vector field $X_{H_0}$ on $N_0$, where $H_0$ is
defined by \eqref{induced_ham}.
\end{thm}

In both cases, the Hamiltonian vector field $X_h$ is not assumed
to be $G$-invariant on $M$. Moreover $X_h|_{M_0}$ may not be
$G$-invariant as well. It is invariant modulo the kernel of
$d\pi_0$, which is sufficient the tools of symplectic reduction
are still applicable.

\begin{dfn}
We shall refer to the passing from $\dot x=X_h|_{M_0}$ to
\begin{equation} \label{REDUCED}
\dot y=X_{H_0}
\end{equation}
as a \emph{partial reduction}.
\end{dfn}

The partial reduction can be seen as a special case of the
symplectic reductions studied in \cite{BT, Li} (see also
\cite{LM}, Ch. III). There, the \emph{symplectic reduction} of a
symplectic manifold $(M,\omega)$ is any surjective submersion
$p:N\to P$ of a submanifold $N\subset M$ onto another symplectic
manifold $(P,\Omega)$, which satisfies $p^*\Omega=\omega|_N$.

\begin{proof}[Proof of theorem \ref{main}]
(i) \; The action of $G$ is generated by Hamiltonian vector fields $X_{\phi_\alpha}$.
Since $h|_{M_{0}}$ is $G$-invariant, for $x\in M_0$ we have
\begin{equation} \label{INVARIANT-EQ}
(dh,X_{\phi_\alpha})=\{h,\phi_\alpha\}=-(d\phi_\alpha,X_h)=0, \quad \alpha=1,\dots,p.
\end{equation}
Thus (\ref{moments}) is an invariant submanifold.

Let $h^*$ be an arbitrary $G$-invariant function that coincides with $h$ on $M_{0}$.
Then $X_{h^*}$ is a $G$-invariant vector field on $M_0$
which project to $X_{H_0}$ \cite{MaWe}:
\begin{equation}\label{projection2}
d\pi_0(X_{h^*})|_x=X_{H_0}|_{y=\pi_0(x)}.
\end{equation}

Let $\delta=h-h^*$. From the condition $\delta|_{M_0}=0$ we can
express $\delta$, in a an open neighborhood of (\ref{moments}), as
$
\delta(x)=\sum_{\alpha=1}^p\delta_\alpha(x)\phi_\alpha(x).
$

Now, let $f$ be a $G$-invariant function on $M$.
Since $\{f,\phi_\alpha\}=0$ (Noether's theorem), we get
$$
\{\delta,f\}|_{M_0}=\biggl(\sum_\alpha\{\delta_\alpha,f\}(\phi_\alpha-\eta_\alpha)
+\sum_\alpha \delta_\alpha\{\phi_\alpha,f\}\biggr)\Big|_{M_0}=0.
$$
Thus the Poisson bracket
$\{\delta,f\}|_{M_0}=-(df,X_\delta)|_{M_0}$ vanish for an
arbitrary $G$-invariant function $f$. In other words
\begin{equation} \label{*}
X_\delta|_x \in T_x (G\cdot x)\cap T_x M_0.
\end{equation}

Combining (\ref{projection2}), $X_h=X_{h^*}+X_\delta$ and (\ref{*})
with the well known identity (e.g., see \cite{OR})
$$
T_x (G\cdot x)\cap T_x M_0=T_x (G_\eta \cdot x)=\ker d\pi_0|_x,
$$
we prove the relation (\ref{projection1}).

(ii) \, Suppose that (\ref{moments}) is an invariant submanifold of (\ref{hamiltonian}).
Then (\ref{INVARIANT-EQ}) holds for $x\in M_0$.
Since the action of $G$ is generated by Hamiltonian vector fields $X_{\phi_\alpha}$,
from (\ref{INVARIANT-EQ}) we get that $h$ is invariant
with respect to the infinitesimal action of $G$.
Therefore we have well defined reduced Hamiltonian function $H_0$ on $N_0$.

Since $M_{0}$ is a closed submanifold of $M$ and the action is proper,
we can find a $G$-invariant function $h^*$ on $M$ which coincides with $h$ on $M_0$.
Now, the relation (\ref{projection1}) follows from the proof of item (i).
\end{proof}

An immediate corollary of theorem \ref{main} is

\begin{cor}
Suppose that the partially reduced system \eqref{REDUCED}
is completely integrable in the non-commutative sense
and $N_0$ is almost everywhere foliated on $r$-dimensional isotropic invariant manifolds,
level sets of integrals $F_i$, $i=1,\dots,\allowbreak\dim N_0-r$.
Then $M_0$ is almost everywhere foliated on $(r+\dim G)$-invariant isotropic manifolds
\begin{equation} \label{M}
\mathcal M_c=\{f_i=F_i\circ \pi_0=c_i \, \vert \, i=1,\dots, {\dim N_0-r} \}
\end{equation}
of the system \eqref{hamiltonian}.
\end{cor}

\begin{rem}
If the partially reduced system (\ref{REDUCED}) is integrable,
then we need ($\dim G$)-additional quadratures for the solving of
$\dot x=X_h|_{M_0}$ (\emph{the reconstruction equations}). If
$r=\frac12 \dim N_0$, the invariant manifolds (\ref{M}) are
Lagrangian.
\end{rem}

\begin{dfn}
We shall say that the Hamiltonian system (\ref{hamiltonian}) is \emph{partially integrable}
if it has invariant relations of the form (\ref{moments})
and that the reduced system (\ref{REDUCED}) is completely integrable
in the non-commutative sense.
\end{dfn}

\begin{exm}
Let $G$ be a torus $\T^p$ and let the reduced flow be completely integrable
in the commutative sense by integrals $F_1,\dots,F_m$, $m=\frac12\dim N_0$.
Consider the regular compact connected component level set of $F_1,\dots,F_m$.
By Liouville's theorem, it is diffeomorphic to a $m$-dimensional torus $\T^m$
with quasi-periodic flow of (\ref{REDUCED}).
Thus the compact connected component $\hat{\mathcal M}_c=\pi^{-1}_0(\T^m)$ of (\ref{M})
is a \emph{torus bundle} over $\T^m$:
\begin{equation} \label{BUNDLE}
\begin{aligned}
\T^p\;\longrightarrow\;\;&\hat{\mathcal M}_c \\
&\;\Big\downarrow{}^{\pi_0} \\
&\;\T^m
\end{aligned}
\end{equation}

Suppose that $f_1,\dots,f_m$ can be extended to commuting $\T^p$-invariant functions
in some $\T^p$-invariant neighborhood $V$ of $\hat{\mathcal M}_c$.
Then, within $V$, $\hat{\mathcal M}_c$ is given by the equations
$$
f_1=c_1,\dots, f_m=c_m, \quad \phi_{1}=0,\dots,\phi_p=0.
$$

From the Noether theorem the functions $\phi_\alpha$
commute with all $\T^n$-invariant functions on $M$
and the following commuting relations hold on $V$:
\begin{gather*}
\{f_a,f_b\}=\{f_a,\phi_\alpha\}=\{\phi_\alpha,\phi_\beta\}=0,\\
 a,b=1,\dots,m, \quad \alpha,\beta=1,\dots,p.
\end{gather*}

Now, as in the case of commutative integrability of Hamiltonian systems $\hat{\mathcal M}_c$
is a Lagrangian torus with tangent space spanned by $X_{f_a}$, $X_{\phi_\alpha}$,
i.e., the bundle (\ref{BUNDLE}) is trivial.

In general, the flow of $\dot x=X_h$ over the torus $\hat{\mathcal M}_c$ is not quasi-periodic:
the vector field $X_{h}$ do not commute
with vector fields $X_{f_1},\dots,X_{f_m},X_{\phi_1},\dots,X_{\phi_p}$
(although Poisson brackets $\{h,f_a\}$, $\{h,\phi_\alpha\}$ vanish on $\hat{\mathcal M}_c$).
\end{exm}

\begin{exm}
To clear up the difference between partial and usual integrability,
let us write down the above problem for the simplest case of $\T^p$-action,
when $M=\T^m\times \R^m \times \T^p\times\R^p=\{(q,p,\varphi,\phi)\}$,
$$
\omega=\sum_{i=1}^m dp_i \wedge dq_i+\sum_{\alpha=1}^pd\phi_\alpha \wedge d\varphi_\alpha,
$$
and the $\T^p$-action is given by translations in $\varphi$ coordinates.
Then the momentum mapping is $\Phi=(\phi_1,\dots,\phi_p)$,
$(M)_0=\{(q,p,\varphi,0)\}$ and the reduced space is
$$
(N_0,\omega_0)=\biggl(\T^m\times \R^m,\sum_{i=1}^m dp_i \wedge dq_i\biggr).
$$

The Hamiltonian equations with Hamiltonian $h(q,p,\varphi,\phi)$ are
\begin{alignat}{3}
\dot q_i&=\frac{\partial h}{\partial p_i},&
\dot p_i&=-\frac{\partial h}{\partial q_i},& i&=1,\dots,m,
\label{c1}\\
\dot\varphi_\alpha&=\frac{\partial h}{\partial \phi_\alpha},\quad&
\dot \phi_i&=-\frac{\partial h}{\partial \varphi_\alpha},\quad&
\alpha&=1,\dots,p.
\label{c2}
\end{alignat}

From (\ref{c2}) it is clear that $M_0$ is an invariant manifold of
the flow if and only if the Hamiltonian $h$ is $\T^p$ invariant on
$M_0$. Let $H_0(q,p)=h(q,p,\varphi,\phi)|_{\Phi=0}$ be the reduced
Hamiltonian. The partially reduced system
\[
\dot q_i= \frac{\partial H_0}{\partial p_i},\quad
\dot p_i=-\frac{\partial H_0}{\partial q_i}, \quad i=1,\dots,m
\]
coincides with non-reduced equation (\ref{c1}) restricted to $M_0$.

If $H_0$ does not depend on action coordinates $p_i$,
the system is quasi-periodic over the tori $p_i=c_i$.
The reconstruction equations in variables $\varphi_\alpha$
are given by (\ref{c2}) for $\Phi=0$.
In general, the partial derivatives $\frac{\partial h}{\partial \phi_\alpha}$
restricted to $M_0$ can depend on $\varphi_\alpha$.
Thus, in general, the system (\ref{c2}) is not solvable
and the motion over invariant Lagrangian tori $p_i=c_i$, $\phi_\alpha=0$ is not quasi-periodic.

If $h$, in addition does not depend on variables $\varphi_\alpha$
on the whole phase space $M$, equations (\ref{c1}), (\ref{c2}) are solvable on $M_0$,
although the system on $M$ may be non-integrable.
\end{exm}

\subsection{Reductions of Additional Symmetries}
Suppose that equations (\ref{moments}) define invariant relations
of the system (\ref{hamiltonian}) with Hamiltonian $h$
and an additional free Hamiltonian action of a connected compact Lie group $K$ is given.
Let
$$
\Psi: M \to \mathfrak k^*
$$
be the corresponding momentum mapping.
Let $\zeta_1,\dots,\zeta_q$ be the base of $\mathfrak k$ and let
\begin{equation} \label{PSI}
\psi_j=(\Psi, \zeta_j), \qquad j=1,\dots,q.
\end{equation}

Further, suppose that the actions of $G$ and $K$ \emph{commute}, that is we have
$$
\{{\phi_\alpha},{\psi_j}\}=0, \qquad \alpha=1,\dots,p,\quad j=1,\dots,q.
$$

The induced $K$-action on $(N_0,\omega_0)$ is also Hamiltonian,
with the momentum mapping $\Psi_0$ satisfying $\Psi|_{M_0}=\Psi_0
\circ \pi_0$.

Denote the projection $M\to M/K$ by $\sigma$.

If $h$ is $K$-invariant, we can reduce the system to the Poisson manifold $M/K$ as well:
\begin{equation} \label{RF}
\dot F=\{F,H\}^K, \qquad F\in C^\infty(M/K),
\end{equation}
where $H$ is the reduced Hamiltonian $H$ ($h=H\circ \sigma$)
and $\{\cdot,\cdot\}^K$ are reduced Poisson bracket.
Since $\phi_\alpha$ are $K$-invariant, we obtain reduced invariant relations
\begin{equation} \label{RIM}
\Phi_\alpha=0, \quad \phi_\alpha=\Phi_\alpha\circ\sigma, \quad \alpha=1,\dots,p
\end{equation}
defining the invariant manifold $M_0/K$ of the reduced flow:
\begin{align*}
M\;\;&\supset\;\;M_0\; \overset{\pi_0}{\longrightarrow}\; N_0=M_0/G\\
^\sigma\Big\downarrow \;\;\;&\qquad\, \Big\downarrow\\
M/K &\supset M_0/K
\end{align*}

On the other side, since $h$ is $K$-invariant,
we have that the reduced Hamiltonian $H_0$ on $N_0$ is also $K$-invariant,
and therefore the momentum mapping $\Psi_0$ is conserved
along the flow of the partially reduced system (\ref{REDUCED}).

Suppose we have additional integrals $\mathcal F$
implying partial integrability with respect to $G$ action.
Then $M_0$ is almost everywhere foliated
on invariant isotropic manifolds $\mathcal M_c$,
level sets of (\ref{PSI}) and integrals $\pi_0^*\mathcal F$ (see Corollary 6.1).

Now, as in the proof of Theorem 5.3 we obtain invariant foliation of $M_0/K$
on manifolds of the form $\mathcal P_c=(K\cdot \mathcal M_c)/K$.
It is clear that $\mathcal P_c$ are isotropic submanifolds
of the symplectic leafs in $(M/K,\{\cdot,\cdot\}^K)$.

\begin{exm}
In the case of Hess--Appelrot system we have $M=T^*SO(3)$,
$G=SO(2)$ (rotations of the body around the vector $\vec r$) and
$K=SO(2)$ (rotations of the body around the vector $\vec\gamma$).

The system is partially integrable with respect to $G$-action:
the reduced system is the spherical pendulum (\ref{spherical_pendulum})
and the reduced phase phase $N_0=T^*S^2$ is foliated on two-dimensional tori,
level sets of the reduced Hamiltonian $H_0$
and momentum mapping of the $G$-action (rotation of the sphere around vector $\vec\Gamma$).
Here the sphere $S^2$ is identified with the positions of the vector $\vec R$.

Therefore, the invariant manifold $M_0\subset T^*SO(3)$ defined by (\ref{HA_m3})
is foliated on $3$-dimensional invariant Lagrangian tori.

The reduction of $K=SO(2)$-symmetry leads to the Euler--Poisson equations (\ref{EP}),
restricted on the invariant set $F_1=(\vec\gamma,\vec\gamma)=1$
(note that $T^*SO(3))/SO(2)\approx\R^3\times S^2$).
The integral $F_3=(\vec m,\vec \gamma)$ is the casimir function.
The invariant $2$-tori (\ref{tori})
can be also seen as a reduction on invariant $3$-tori of $M_0$.

Note that the system (\ref{EP}), (\ref{conditions}), (\ref{invariant_relation})
has no invariant measure.
Namely, by the Euler--Jacobi theorem (e.g., see \cite[p.~131]{AKN}),
the invariant measure together with the integrals (\ref{integrals11}),
(\ref{invariant_relation}) would imply solvability of the system by quadratures.
\end{exm}

\subsection{ Invariant Measure}
By the Liouville theorem, the Hamiltonian flow (\ref{hamiltonian})
preserves the canonical measure $\Omega=\omega^{\dim M/2}$.
Let $L_{X_h}$ denotes the Lie derivative with respect to the flow (\ref{hamiltonian}).
Then
\begin{equation} \label{Liouville}
L_{X_h}\Omega=0.
\end{equation}

We define the restriction of $\Omega$ to $M_0$ as follows (e.g., \cite[p.~28]{AKN}).
Let $\Theta$ be an arbitrary volume form on $M_\eta$.
We can cover an open neighborhood of $M_\eta$
with charts $U_i$ having local coordinate systems of the form
$
(z^i,\phi)=(z^i_1,\dots,z^i_m,\phi_1,\dots,\phi_p),
$
where $\phi_\alpha$ are given by (\ref{moments}). On $U_i$ we have
$$
\Omega^i(z^i,\phi)
=\mu^i(z^i,\phi)\,d\phi_1\wedge d\phi_2\wedge\dots\wedge d\phi_n\wedge\Theta^i(z^i),
$$
for some smooth nonvanishing function $\mu^i$. Then the
restriction of $\Omega$ to $M_\eta$ is the form $\hat\mu\,\Theta$
locally given by
$\mu^i(z^i,\phi)\,\Theta^i(z^i)|_{\phi_\alpha=0}$.

The functions $\phi_\alpha$ are particular integrals of the equations (\ref{hamiltonian}).
Therefore the time derivative of $\phi_\alpha$,
in an open neighborhood of $M_\eta$, is of the form
\begin{equation} \label{dot_g}
\dot \phi_\alpha(x)=\{\phi_\alpha,h\}=\sum_{\beta}\psi_{\alpha\beta}(x)\phi_\alpha,
\end{equation}
where $\psi_{\alpha\beta}$ are smooth functions.
From (\ref{dot_g}) we get
$$
L_{X_h}d\phi_\alpha
=\sum_{\beta}\left(\psi_{\alpha\beta}d\phi_\alpha+\phi_\alpha d\psi_{\alpha\beta}\right).
$$
Whence, at the points of $M_0 \cap U_i$ we have
\begin{equation} \label{trace}
L_{X_h}\Omega^i=d\phi_1\wedge\dots\wedge d\phi_p\wedge L_{X_h}(\mu^i\,\Theta^i)
+\mathrm{tr}\,\psi\,\mu^i\,d\phi_1\wedge d\phi_2\wedge \dots \wedge d\phi_p\wedge\Theta^i,
\end{equation}
where $\mathrm{tr}\,\psi$ is the trace of the matrix $\psi_{\alpha\beta}$.
The relations (\ref{Liouville}), (\ref{trace}) lead to the following proposition:

\begin{prop}
\emph{(i)}
The flow \eqref{hamiltonian} restricted to the invariant manifold \eqref{moments}
preserve the restriction of $\Omega$ to $M_0$
if and only if $\mathrm{tr}\,\psi(x)=0$, for $x\in M_0$.
In particular, if the Hamiltonian $h$ is a $G$-invariant function,
then $\hat\mu\,\Theta$ is an invariant volume form.

\emph{(ii)}
If in addition we have the Hamiltonian action of the compact Lie group $K$
as described above and the flow \eqref{hamiltonian}
restricted to the invariant manifold \eqref{moments}
preserve the restriction of $\Omega$ to $M_0$,
then the reduced flow \eqref{RF} restricted to the invariant manifold \eqref{RIM}
also has the invariant measure.
\end{prop}

Item (ii) follows from the following general statement (e.g., see \cite{FeJo}):

\begin{lem}
Suppose a compact group $K$ acts freely on a manifold $N$ with local coordinates $z$,
and there is a $K$-invariant dynamical system $\dot z=Z(z)$ on $N$.
If this system has an invariant measure (which is not necessarily $K$-invariant),
then the reduced system on the quotient manifold $N/K$ also has an invariant measure.
\end{lem}

\subsection{Partial Lagrange--Routh Reductions}
Let $(Q,\kappa,v)$ be a natural mechanical system.
Let $G$ be a connected Lie group acting freely and properly on $Q$ and $\rho:
Q\to Q/G$ be the canonical projection.
As above we define vertical and horizontal distribution
by (\ref{vert}) and (\ref{horizontal}), respectively
(now we do not suppose that $\kappa$ is $G$-invariant).

Since $\kappa_q(\mathcal H_q)=\ann \mathcal V_q$,
we see that $\mathcal H$ is invariant with respect to the ``twisted" $G$-action
\begin{equation} \label{action}
{g}\diamond(q,X)
=(g\cdot q,\,\kappa_{g\cdot q}^{-1}\circ(dg^{-1})^*\circ\kappa_q(X)),\quad X\in T_q Q,
\end{equation}
that is the pull-back of canonical symplectic $G$-action on $T^*Q$ via metric $\kappa$:
\[
\begin{CD}
TQ @>{\kappa}>> T^*Q \\
@V{g\diamond}VV  @VV{g\cdot}V \\
TQ @>>{\kappa}> T^*Q \\
\end{CD}
\]

From Theorem \ref{main} we obtain (see \cite{Jo2})

\begin{thm}
\emph{(i) \ (Partial Noether theorem)}
The horizontal distribution \eqref{horizontal}
is an invariant submanifold of the Euler--Lagrange equations \eqref{Lagrange}
if and only if the potential $v$ and the restriction $\kappa_\mathcal H$
of the metric $\kappa$ to $\mathcal H$
are $G$-invariant with respect to the action \eqref{action}.

\emph{(ii) \ (Partial Lagrange--Routh reduction)}
If $\mathcal H$ is an invariant submanifold of the system $(Q,\kappa,v)$,
then the trajectories $q(t)$ with velocities $\dot q(t)$ that belong to $\mathcal H$
project to the trajectories $b(t)=\pi(q(t))$ of the natural mechanical system $(Q/G,K,V)$
with the potential $V(\pi(q))=v(q)$ and the metric $K$ obtained from $\kappa_\mathcal H$
via identification $\mathcal H/G\approx T(Q/G)$.
\end{thm}

Note that when $\kappa$ is $G$-invariant, the twisted $G$-action (\ref{action})
coincides with usual $G$-action: $g\cdot(q,X)=(g\cdot q,dg(X))$
and the induced metric $K$ is the submersion metric.
In this case we have theorem \ref{LR}.

The partial reduction of the classical Hess--Appel'rot problem
is the spherical pendulum on $S^2$ (\ref{spherical_pendulum}).
The partial reduction of the $n$-dimensional variant
of the system given in \cite{Jo2} is a spherical pendulum on $S^{n-1}$.
It appears that the $4$-dimensional variant of the Hess--Appel'rot system
given by Dragovi\'c and Gaji\'c \cite{DrGa2}
can be also considered within the framework of partial reductions,
where the reduced system is a pendulum system on oriented Grassmannian variety $G^+(4,2)$
of oriented two-dimensional planes in $\R^4$ (see \cite{Jo2}).

The spherical pendulum as well as the pendulum system on Grassmannian variety $G^+(4,2)$
are completely integrable
(the complete integrability of the later system
follows from complete integrability of the pendulum type systems
on adjoint orbits of compact Lie groups established in \cite{BJ5}).

Therefore it is natural to define

\begin{dfn}
We shall say that a natural mechanical system $(Q,\kappa,v)$
satisfies \emph{geometrical Hess--Appel'rot conditions}
if it is partially integrable, i.e, it has an invariant relation (\ref{horizontal})
and the partially reduced system $(Q/G,K,V)$ is completely integrable.
\end{dfn}

\begin{exm}
Let $(Q,\kappa_*)$ be a compact Riemannian manifold with a
$G$-invariant metric. Let $h^*=\frac12(p,\kappa_*^{-1} p)$. The
reduction of the geodesic flow $X_{h^*}$ to
$T^*(Q/G)\approx(T^*Q)_0/G$ is the geodesic flow of the submersion
metric $K$ on $Q/G$. Let $\pr_{\ann\mathcal V}$ be the orthogonal
projection to $\ann\mathcal V_q$ in $T_q^*Q$ with respect to the
metric $\kappa_*^{-1}$, and let $\A$ be an arbitrary metric on the
cotangent bundle. Then $h=h^*+\frac12(\pr_{\ann\mathcal V}
p,\A_qp)$ will be the Hamiltonian function of the certain metric
that we shall denote by $\kappa$. Since $h=h^*|_{(T^*Q)_0}$ we can
perform the partial Lagrange--Routh reduction. Note that the
metrics $\kappa$ and $\kappa_*$ induce the same metric $K$ on
$Q/G$, but their horizontal distributions $\mathcal H$ and
$\mathcal H^*$ are different:
$$
\mathcal H_q=\kappa^{-1}\circ \kappa_*|_q (\mathcal H_q^*).
$$

Further, let us suppose that the geodesic flow of the submersion metric
is completely integrable.
Then the geodesic flow of the metric $\kappa$ is, in general, non-integrable
but has the following interesting property:
$(T^*Q)_0$ is almost everywhere foliated by compact invariant Lagrangian submanifolds.
\end{exm}

\subsection{Complete Integrability on Invariant Manifolds}
The notion of partial integrability introduced here differs from those
based on the Poincare-–Lyapo\-unov-–Nekhoroshev theorem \cite{N2, GG, N3}
(see also \cite{GMS, MPY}), where a Hamiltonian system (\ref{hamiltonian})
restricted to an invariant submanifold $N\subset M$ of lower dimension
is \emph{completely integrable},
i.e., $N$ is filled with periodic or quasi-periodic trajectories of (\ref{hamiltonian}).

Let $(M,\omega)$ be a $2n$-dimensional symplectic manifold.
Let $f_1,\dots,f_k$ be real, Poisson commuting functions on $M$, where $k<n$.
Assume that there exists a compact connected $k$-dimensional manifold $T^k$
invariant under all Hamiltonian flows $X_{f_i}$,
and that $dF$ has rank $k$ on all points of $T^k$, where
$$
F:M \to \R^k, \qquad F=(f_1,\dots,f_k).
$$
Then, as in the Liouville theorem \cite{Ar},
$T^k$ is diffeomorphic to $k$-dimensional torus and $T^k$
is a submanifold of $F^{-1}(\beta_0)$ for some $\beta_0\in\R^k$.

Take a $k$-parameter family of local $(2n-k)$-manifolds $\Sigma_m$,
passing through $m\in T^k$ and transversal to both $T^k$ and $F^{-1}(\beta_0)$ at $m$.
In a neighborhood of $m\in T^k$ consider also the local foliation $\mathcal L^m$
generated by vector fields $X_{f_i}$.

Any homotopy class $\gamma$ in $\pi_1(T^k)$ can be realized as the orbit
of some vector field obtained as a linear combination of $X_{f_i}$:
$$
X_c=\sum_{i=1}^k c_i X_{f_i}
$$
such that $X_c$ has closed trajectories with period 1 on $T^k$.

If we consider the time-one flow of points $p\in\Sigma_m$ under $X_c$,
this defines a map $\Theta_c^m$ from $\Sigma_m$ to $M$.
By construction $\Theta_c^m(m)=m$,
while in general $\Theta_c^m(p)$ can fail to be in $\Sigma_m$.
However, in an appropriate neighborhood of $m$,
the leaf of the foliation $\mathcal L^m$ through $\Theta_c(p)$
intersects $\Sigma_m$ in a unique point $p'$.
In this way we have well defined map
$\Psi_c^m: \Sigma_m \longrightarrow \Sigma_m$, $\Psi_c^m(p)=p'$
(for more details see \cite{N2, GG}).
Following \cite{GG}, we call this the \emph{Poincare--Nekhoro\-shev} map.
It is based at a point $m$ and depends on the constants $c_1,\dots,c_k$,
i.e., on the homotopy class $\gamma$.

With the above notations, we can state the following theorem:

\begin{thm} [Nekhoroshev \cite{N2}]
Suppose that the spectrum of the linear part of the Poincare--Nekhoroshev map $\Psi_c^m$
associated to $X_c$ does not include the unity, for some $m\in T^k$.
Then, in a neighborhood $U$ of $T^k$ in $M$,
there is a symplectic submanifold $N$ which is fibered over a domain $B\subset\R^k$
with fibers been $k$-tori $T^k_\beta N \cap F^{-1}(\beta)$, $\beta\in B$
invariant under the flows of $X_{f_i}$, $i=1,\dots,k$.
\end{thm}

The theorem extends the Poincare--Lyapounov theorem on persistence
of periodic trajectories (invariant $S^1$ orbits)
 to the case of higher dimensional tori.
Also it extends the Liouville--Arnold theorem
by showing that one can find partial action-angle coordinates on these invariant tori,
i.e. one can define coordinates $(I_j,\varphi_j;p_a,q_a)$, with $j=1,\dots,k,$,
$a= 1,\dots,n-k$, in a open neighborhood of $T^k$
so that the invariant tori correspond to $p=q=0$
and are parametrized by the value of the action coordinates $I_j$ (see \cite{N2,N3}).

\subsection*{Acknowledgments}
I am very grateful to Alexey V.~Bolsinov
(most of the survey is the result of our collaboration),
Vladimir Dragovi\'c, Yuri N.~Fedorov, Borislav Gaji\'c and Milena Radnovi\'c
for support and many interesting discussions on the subject.
The survey is based on the talk given at the Conference
\emph{MM-VII Symmetries and Mechanics}, Novi Sad 2007.
I would like to thank the Organizers for kind hospitality.
Also, I would like to thank the referee for useful suggestions.
The research is supported by the Serbian Ministry of Science,
Project 144014, \emph{Geometry and Topology of Manifolds and Integrable Dynamical Systems}.

\end{document}